\documentclass[letterpaper, reqno, 11pt]{amsart}
\usepackage{amssymb, amscd, latexsym, amsfonts, mathrsfs}
\usepackage[english]{babel}
\RequirePackage[utf8]{inputenc}

\usepackage{enumerate, cite, calc, xstring, hyperref, cancel}
\usepackage{float, wrapfig, sidecap, multicol, needspace, booktabs, simplewick}

\usepackage[T1]{fontenc}
\usepackage[all]{xy}
\usepackage{graphicx, xcolor, url}
\usepackage[top=2.75cm, bottom=2.75cm, left=2.75cm, right=2.75cm]{geometry}

\usepackage[T1]{fontenc}
\usepackage[scaled=.90]{helvet} % \textss style substitute
\usepackage{dsfont, upgreek, mathabx, yfonts}
\usepackage{relsize}

\newtheorem{theorem}{Theorem}[section]
\newtheorem{proposition}[theorem]{Proposition}
\newtheorem{corollary}[theorem]{Corollary}
\newtheorem{lemma}[theorem]{Lemma}

\theoremstyle{definition}
\newtheorem{definition}[theorem]{Definition}
\newtheorem{example}[theorem]{Example}

\theoremstyle{remark}
\newtheorem{remark}[theorem]{Remark}

\numberwithin{equation}{section}

\newcommand{\tr}{\mathrm{tr}\,}

\newcommand{\mr}[1]{\mathrm{#1}}
\newcommand{\mc}[1]{\mathcal{#1}}

\newcommand{\mbb}[1]{\mathbb{#1}}

\newcommand{\bs}[1]{\boldsymbol{#1}}

\newcommand{\cx}{\bar{x}}

\newcommand{\bze}{\bs{0}}

\newcommand{\bb}{\bs{s}}
\newcommand{\bx}{\bs{x}}
\newcommand{\by}{\bs{y}}
\newcommand{\bz}{\bs{z}}
\newcommand{\bu}{\bs{u}}
\newcommand{\bv}{\bs{v}}
\newcommand{\bw}{\bs{w}}
\newcommand{\bV}{\bs{V}}
\newcommand{\bH}{\bs{H}}
\newcommand{\VV}{\mbb{V}}

\newcommand{\Gr}[2]{\mr{Gr}({#1},{#2})}
\newcommand{\projj}[2]{\mr{proj}_{#2}({#1})}

\newcommand{\bE}{\bs{e}}
\newcommand{\bN}{\bs{n}}
\newcommand{\bM}{\bs{m}}
\newcommand{\bR}{\bs{R}}
\newcommand{\Rie}[4]{\langle\bR({#1},{#2}){#3},\,{#4}\rangle}
\newcommand{\Ric}[2]{\bs{\mc{R}ic}({#1},{#2})}
\newcommand{\SH}{\bs{\hat{S}_H}}
\newcommand{\oS}{\bs{\hat{S}}}
\newcommand{\oR}{\bs{\hat{\mc{R}}}}

\newcommand{\cR}{\mc{R}}

\newcommand{\II}{\bs{\mr{II}}}
\newcommand{\III}{\bs{\mr{III}}}

\newcommand{\metric}[2]{\langle\, {#1},\,{#2}\,\rangle}

\newcommand{\Cyll}{\mr{Cyl}_p}

\renewcommand{\tilde}{\widetilde}
\renewcommand{\hat}{\widehat}
\renewcommand{\bar}{\overline}

\newcommand{\dd}{\partial}

\newcommand{\ga}{\gamma}

\newcommand{\del}{\delta}

\newcommand{\al}{\alpha}
\newcommand{\bet}{\beta}
\newcommand{\vep}{\varepsilon}
\newcommand{\lbd}{\lambda}

\newcommand{\ka}{\kappa}

\newcommand{\diag}{\mathrm{diag}}
\def\RR{{\mathbb R}}

\def\SS{{\mathbb S}}

\def\cS{{\mathcal S}}

\def\cO{{\mathcal O}}
\def\cM{{\mathcal M}}

\def\cB{{\mathcal B}}

\newcommand{\Id}{\mathrm{Id}}

\begin{document}

\title{Integral Invariants from Covariance Analysis of Embedded Riemannian Manifolds}

\author[J. \'Alvarez-Vizoso]{Javier \'Alvarez-Vizoso}
\author[M. Kirby]{Michael Kirby}
\author[C. Peterson]{Chris Peterson}
\address{Department of Mathematics, Colorado State University, Fort Collins, CO, USA}
\email{alvarez@math.colostate.edu, kirby@math.colostate.edu, peterson@math.colostate.edu}

\date{\today}

\maketitle

%%%%%%%%%%%%%%%%%%%%%%%%%%%%%%%%%%%%%%%%%%%%%%%%%%%%%%%%%%%%%%%%%%%%%%
%%%%%%%%%%%%%%%        ABSTRACT
%%%%%%%%%%%%%%%%%%%%%%%%%%%%%%%%%%%%%%%%%%%%%%%%%%%%%%%%%%%%%%%%%%%%%%

% REQUIRED
\begin{abstract}
Principal Component Analysis can be performed over small domains of an embedded Riemannian manifold in order to relate the covariance analysis of the underlying point set with the local extrinsic and intrinsic curvature. We show that the volume of domains on a submanifold of general codimension, determined by the intersection with higher-dimensional cylinders and balls in the ambient space, have asymptotic expansions in terms of the mean and scalar curvatures. Moreover, we propose a generalization of the classical third fundamental form to general submanifolds and prove that the eigenvalue decomposition of the covariance matrices of the domains have asymptotic expansions with scale that contain the curvature information encoded by the traces of this tensor. In the case of hypersurfaces, this covariance analysis recovers the principal curvatures and principal directions, which can be used as descriptors at scale to build up estimators of the second fundamental form, and thus the Riemann tensor, of general submanifolds.
\end{abstract}

%%%%%%%%%%%%%%%%%%%%%%%%%%%%%%%%%%%%%%%%%%%%%%%%%%%%%%%%%%%%%%%%%%%%%%
%%%%%%%%%%%%%%%        TOC customization

% use [4] if hyperref is not in use, [5] otherwise
\DeclareRobustCommand{\SkipTocEntry}[5]{}

\makeatletter
\def\@tocline#1#2#3#4#5#6#7{\relax
\ifnum #1>\c@tocdepth % then omit
  \else 
    \par \addpenalty\@secpenalty\addvspace{#2}% 
\begingroup \hyphenpenalty\@M
    \@ifempty{#4}{%
      \@tempdima\csname r@tocindent\number#1\endcsname\relax
 }{%
   \@tempdima#4\relax
 }%
 \parindent\z@ \leftskip#3\relax \advance\leftskip\@tempdima\relax
 \rightskip\@pnumwidth plus4em \parfillskip-\@pnumwidth
 #5\leavevmode\hskip-\@tempdima #6\nobreak\relax
 \ifnum#1<0\hfill\else\dotfill\fi\hbox to\@pnumwidth{\@tocpagenum{#7}}\par
 \nobreak
 \endgroup
  \fi}
\makeatother

\tableofcontents

%%%%%%%%%%%%%%%%%%%%%%%%%%%%%%%%%%%%%%%%%%%%%%%%%%%%%%%%%%%%%%%%%%%%%%
%%%%%%%%%%%%%%%%%%%%%%%%%%%%%%%%%%%%%%%%%%%%%%%%%%%%%%%%%%%%%%%%%%%%%%
%%%%%%%%%%%%%%%			0. INTRODUCTION
%%%%%%%%%%%%%%%%%%%%%%%%%%%%%%%%%%%%%%%%%%%%%%%%%%%%%%%%%%%%%%%%%%%%%%
%%%%%%%%%%%%%%%%%%%%%%%%%%%%%%%%%%%%%%%%%%%%%%%%%%%%%%%%%%%%%%%%%%%%%%

\section{Introduction}\label{sec:intro}

Local integral invariants based on Principal Component Analysis have been introduced in the literature as theoretical tools to perform Manifold Learning and Geometry Processing of low-dimensional submanifolds, like curves and surfaces in space. Curvature descriptors obtained in this way serve as feature and shape estimators at scale whose numerical implementation takes advantage of the benefits of employing integrals instead of differentials when only a discrete sample of points is available. Integral invariants provide a theoretical link between the statistical covariance analysis of the underlying point-set of a domain and the differential-geometric invariants at a point of the domain inside the manifold. In particular, intersecting the submanifold with a ball in the ambient space cuts out a subdomain whose covariance matrix has an eigenvalue decomposition that asymptotically expands with the scale of the ball. The geometric interpretation of this analysis lies in the fact that the first and second terms of the eigenvalue series encode the curvature information of the submanifold at the center of the ball. 

Integral invariants have been introduced and used in Computer Graphics and Geometry Processing by \cite{connolly1986},\cite{cazals2003a, cazals2003b}, \cite{clarenz2003, clarenz2004}, \cite{manay2004, manay2006}. The integral invariant viewpoint via Principal Component Analysis has been introduced and studied theoretically and numerically \cite{alliez2007}, \cite{berkmann1994}, \cite{clarenz2003, clarenz2004}, \cite{hoppe1992}, \cite{manay2004}, \cite{pottmann2006}, with a focus on curves and surface, in order to process discrete samples of points to determine features and detect shapes at scale, and study stability with respect to noise \cite{lai2009}, \cite{pottmann2007, pottmann2008}. Voronoi-based covariance matrices have been also been of interest, \cite{merigot2009, merigot20011}.
The eigenvalue decomposition of covariance matrices of spherical intersection domains was also introduced by \cite{broomhead1991local}, \cite{solis1993, solis2000}, in order to obtain local adaptive Galerkin bases for the invariant manifold of large-dimensional dynamical systems. For curves, the Frenet-Serret frame is recovered in the scale limit, and ratios of the covariance matrix eigenvalues provide descriptors at scale of the generalized curvatures \cite{alvarez2017}, but the tools needed to study the curve case are significantly different due to the fact that one-dimensional submanifolds have only extrinsic curvature.

In the present work we generalize to embedded Riemannian manifolds of general codimension the recent study of PCA integral invariants of hypersurfaces \cite{alvarez2018a}, that followed the theoretical study of surfaces in \cite{pottmann2007}, with the purpose of obtaining analogous asymptotic formulas between eigenvalues of covariance matrices and curvature, as it was found for curves in \cite{alvarez2017}. We shall also introduce a generalization to general codimension of the third fundamental form in order to encapsulate all the curvature information hidden in the covariance analysis. Our main result shows how the eigenvalue decomposition of the covariance of cylindrical and spherical intersection domains has the first two orders of the asymptotic expansion given in terms of the dimension, and the extrinsic and intrinsic curvature encoded in the traces of the third fundamental form, with limit eigenvectors playing the role of generalized principal directions.

Geodesic balls inside manifolds have asymptotic series for their intrinsic volume given as corrections to the Euclidean ball completely determined by intrinsic scalar curvature invariants \cite{gray1979}. In our case, the domains of integration depend on the embedding of the submanifold so the extrinsic curvature will play a crucial role in the volume corrections, as in \cite{hulin2003}. Normal coordinates via the exponential map are naturally used to do geometric measurements needed for probability and statistics from an intrinsic perspective inside Riemannian manifolds, e.g. \cite{pennec1999, pennec2006}. The generalized definition of integral invariants makes use of the exponential map in the ambient manifold to make measurements over the underlying point-set of a submanifold.

The structure of the paper is as follows: in section \S\ref{sec:IntInv} we propose a general definition of integral invariants in the context of general Riemannian submanifolds by use of the exponential map, along with the two types of kernel domains on which we will perform the PCA. In section \S\ref{sec:IIIform}, the study of the geometry of submanifolds via the second fundamental form is briefly reviewed and the classical third fundamental form is generalized to submanifolds of general codimension. In section \S\ref{sec:cylCov} we compute the volume, barycenter and covariance matrix of a cylindrical domain inside an embedded submanifold; in particular, we show that the scaling of the eigenvalues of the covariance matrix singles out the tangent and normal spaces of the manifold at the point by the span of the corresponding limit eigenvectors, and how the next-to-leading order term in the asymptotic series of the eigenvalues is determined by the eigenvalues of the tangent and normal traces of the third fundamental form. In section \S\ref{sec:SphCov} an analogous analysis is carried out for the domain determined by the intersection of a ball in ambient space  with the manifold, which introduces considerable correction terms with respect to the previous case. This leads to an eigenvalue decomposition of the covariance matrix with tangent part given in terms of the Weingarten operator corresponding to the mean curvature vector. Finally, in section \S\ref{sec:descrip} we obtain the limit ratios of the eigenvalues in terms of this curvature information, and invert the asymptotic series to get descriptors at scale for the case of hypersurfaces, where the second and third fundamental forms are completely given by the principal curvatures and principal directions.

These results show how Principal Component Analysis can be carried out on a general embedded Riemannian submanifold to probe its local geometry. It establishes the relationship between the statistical covariance analysis of the underlying point-set of the manifold and the classical differential-geometric curvature via the third fundamental form. Applying the integral invariant approach to hypersurfaces provides a method to build multi-scale descriptors of curvature also for the case of general codimension.

%%%%%%%%%%%%%%%%%%%%%%%%%%%%%%%%%%%%%%%%%%%%%%%%%%%%%%%%%%%%%%%%%%%%%%
%%%%%%%%%%%%%%%%%%%%%%%%%%%%%%%%%%%%%%%%%%%%%%%%%%%%%%%%%%%%%%%%%%%%%%
%%%%%%%%%%%%%%%			1. INTEGRAL INVARIANTS & DESCRIPTORS
%%%%%%%%%%%%%%%%%%%%%%%%%%%%%%%%%%%%%%%%%%%%%%%%%%%%%%%%%%%%%%%%%%%%%%
%%%%%%%%%%%%%%%%%%%%%%%%%%%%%%%%%%%%%%%%%%%%%%%%%%%%%%%%%%%%%%%%%%%%%%

\section{PCA Integral Invariants of Riemannian Submanifolds}\label{sec:IntInv}

In our context, integral invariants are local integrals over domains of a submanifold determined by intersection with objects in the ambient space, like spheres. Two such integrals are the volume of the domain and the point in the ambient manifold that represents the center of mass of the region. A more interesting object is the covariance matrix obtained by integrating the relative covariance of the degrees of freedom of the points in the domain, i.e., the products of the coordinates of the points with respect to a chosen frame. In order to get a frame independent integral invariant, one takes the eigenvalue decomposition of the covariance matrix. Since the kernel domains have a natural scale, e.g., the radius of the sphere, it is useful to think of them as a matrix-valued function of scale at every point. Therefore, these integral invariants correspond to eigenvalues and eigenvectors that can be interpreted respectively as a set of scalar and frame-valued functions of scale at every point. The study of covariance matrices in order to obtain adapted frames of general submanifols was studied for example in \cite{broomhead1991local} and \cite{solis1993,solis2000}, whereas the integral invariant approach was developed in detail to extract the curvature information of surfaces in space in \cite{pottmann2007}.

In order to do this type of Principal Component Analysis on a general Riemannian submanifold and generalize local integral invariants, definitions using Cartesian coordinates must naturally be promoted to Riemann normal coordinates \cite{chavel2006}, \cite{oneil1983}.
If the $n$-dimensional submanifold $\cM^{(n)}$ sits inside an ambient Riemannian manifold $(\mc{N}^{(n+k)},g)$, the curves in $\mc{N}$ that generalize the axis used in $\RR^{n+k}$ are the geodesic curves $\ga_{\bv}(t)$ and these always exist and are unique locally at any point $p\in\mc{N}$ and direction $\bv$. Given an orthonormal frame in $T_p\cM\oplus N_p\cM$, the geodesics tangent to each of the vectors will trace out generalized coordinate axis in $\mc{N}$ that, through the exponential map will uniquely specify any point in a local neighborhood around $p$. Assuming $\mc{N}$ is geodesically complete to simplify the exposition, the exponential map collects all geodesics starting at $p$ by mapping straight lines through the origin in $T_p\mc{N}\cong\RR^{n+k}$ to geodesics through $p$:
$$
\exp_p: T_p\mc{M}\rightarrow\mc{N}\quad \text{ such that }\quad \exp_p(t\bv) = \ga_{t\bv}(1)=\ga_{\bv}(t).
$$
At any point $p$ there is a neighborhood $\tilde{\mc{U}}$ of $\bze$ in $T_p\mc{N}$ where $\exp$ is a diffeomorphism onto a neighborhood $\mc{U}$ of $p$ in $\mc{N}$. From this, for star-shaped $\tilde{\mc{U}}$, there is also a unique geodesic $\ga(t)$ connecting $p$ and any other point $q\in\mc{U}$ such that the tangent $\ga'(0)=\exp^{-1}_p(q)$. Moreover, the arclength of $\ga$ between the two points, i.e. the distance $d(p,q)$ between them determined by the metric $g$, is the length of the tangent vector representation through this map, $d(p,q)  = \| \exp^{-1}_p(q) \|$.
These normal neighborhoods allow the parametrization of points using the geodesic distances tangent to a given frame $\{\bE_\mu\}_{\mu=1}^{n+k}$ at $p$. The injectivity radius $r_p$ is the radius of the largest ball $B_{\bze}(\vep)$ in $T_p\mc{N}$ where $\exp$ is a diffeomorphism, so $\cB_p(r_p)=\exp_p(B_{\bze}(r_p))$ is the largest ball in $\mc{N}$ created by radial geodesics of the same length around $p$ where normal coordinates are well-defined. In fact $r_p>0$ always. Since our main theorems \ref{MainTh} and \ref{MainTh2} are asymptotic results with scale, in a general Riemannian manifold one could always use normal coordinates to study domains of submanifolds small enough so that they can be mapped to Euclidean space, thus, we propose the following general definition of PCA integral invariants in a general Riemannian manifold.

\begin{definition}
	Let $D$ be a measurable domain in a Riemannian manifold $(\mc{N},g)$ such that $D\subset \cB_p(r_p)$ for some point $p\in\mc{N}$, The \emph{integral invariants} associated to the moments of order 0, 1 and 2 of the geodesic coordinate functions of the points of $D$ with respect to $p$ are: \\
	the volume
	\begin{equation}
		V(D) = \int_D 1\; \text{dVol},
	\end{equation}
	the barycenter
	\begin{equation}
		\bb(D) = \frac{1}{V(D)}\int_D [\exp_p^{-1}(q)] \; \text{dVol},
	\end{equation} 
	and the eigenvalue decomposition of the covariance matrix:
	\begin{equation}
		C(D) =\int_D [\exp_{p}^{-1}(q)]\otimes [\exp_p^{-1}(q)] \; \text{dVol}.
	\end{equation} 
	Here dVol is the measure on $D$, restriction of the measure of $\mc{N}$ induced by the metric $g$, and the tensor product is to be understood as the outer product of the components of the $\exp^{-1}$ map in a chosen orthonormal basis of $T_p\mc{N}$. The reference point of the covariance matrix is often chosen to be the barycenter $\exp_p(\bb)$ instead of $p$.
\end{definition}

The two types of domains that we shall study are regions in a submanifold $\cM\subset\mc{N}$ determined by the intersection with a ball and a cylinder. Using the exponential map one can define such intersections by mapping Euclidean balls and higher-dimensional cylinders in $T_p\mc{N}$ to their geodesic generalizations in the ambient manifold $\mc{N}$.

\begin{definition}
	The \emph{spherical component} of radius $\vep\leq r_p$, at a point $p$ of a submanifold $\cM$ of a Riemannian manifold $\mc{N}$ is the domain given by:
\begin{equation}
	D_p(\vep):=\cM\cap \{q\in\mc{N}: \|\exp_p^{-1}(q)\|\leq\vep\leq r_p \}.
\end{equation}
\end{definition}

An element $\VV$ in the Grassmannian $\Gr{m}{n+k}$ is an $m$-dimensional linear subspace of $\RR^{n+k}$. Fixing a point and $m$-dimensional ball inside $\VV$, the standard three dimensional cylinder over the $xy$-plane can be generalized to an $\VV$-cylinder by taking all points in the ambient space that project down onto the ball inside $\VV$.

\begin{definition}
	The \emph{cylindrical component} of radius $\vep\leq r_p$, at a point $p$ of a submanifold $\cM$ of a Riemannian manifold $\mc{N}$ over the $m$-plane $\VV\in\Gr{m}{n+k}$, is the $\VV$-cylinder intersection:
\begin{equation}
\Cyll(\vep,\VV):=\cM\cap \{q\in\mc{N}: \|\projj{\exp_p^{-1}(q)}{\VV}\|\leq\vep\leq r_p \}, 
\end{equation}
where $\projj{\cdot}{\VV}$ is the orthogonal projection onto $\VV$ as a linear subspace of $T_p\mc{N}$. We shall write $\Cyll(\vep)$ when $\VV=T_p\cM$ is assumed.
\end{definition}

We will compute these integral invariants for embedded submanifolds in Euclidean ambient space, $\mc{N}=\RR^{n+k}$, where $\exp^{-1}_p(q)=\bs{q}-\bs{p}$ as vectors and the tensor product recovers the common  definition of PCA integral invariants studied in the literature. The points $q\in D$ are then parametrized by a vector $\bs{X}$ such that the barycenter is the center of mass
	\begin{equation}
		\bb(D) = \frac{1}{V(D)}\int_D \bs{X} \; \text{dVol},
	\end{equation} 
	and the the covariance matrix can be interpreted as analogous to a moment of inertia matrix, which for the cylindrical component shall be taken with respect to the center $p$, following the convention and motivation of \cite{solis2000}, 
	\begin{equation}
		C(\Cyll(\vep)) =\int_{\Cyll(\vep)} (\bs{X}-\bs{p})\otimes (\bs{X}-\bs{p}) \; \text{dVol},
	\end{equation} 
	whereas for the spherical component the covariance matrix shall be taken with respect to the barycenter following \cite{pottmann2007},
		\begin{equation}
		C(D_p(\vep)) =\int_{D_p(\vep)} (\bs{X}-\bb(D_p(\vep)))\otimes (\bs{X}-\bb(D_p(\vep))) \; \text{dVol}.
	\end{equation} 

Without loss of generality, these definitions could have been normalized by the volume of the domain to make the integral measure become a probability density and thus make the matrices actual statistical covariances.

%%%%%%%%%%%%%%%%%%%%%%%%%%%%%%%%%%%%%%%%%%%%%%%%%%%%%%%%%%%%%%%%%%%%%%
%%%%%%%%%%%%%%%%%%%%%%%%%%%%%%%%%%%%%%%%%%%%%%%%%%%%%%%%%%%%%%%%%%%%%%
%%%%%%%%%%%%%%%			2. THIRD FUNDAMENTAL FORM
%%%%%%%%%%%%%%%%%%%%%%%%%%%%%%%%%%%%%%%%%%%%%%%%%%%%%%%%%%%%%%%%%%%%%%
%%%%%%%%%%%%%%%%%%%%%%%%%%%%%%%%%%%%%%%%%%%%%%%%%%%%%%%%%%%%%%%%%%%%%%

\section{Third Fundamental Form of a Riemannian Submanifold}\label{sec:IIIform}

For a complete analysis of the the geometry of Riemannian submanifolds see \cite{chavel2006}, \cite{kobayashi1969}, \cite{oneil1983}, \cite{spivak1999}. 

Let $(\cM,g)$ be an $n$-dimensional manifold isometrically embedded in an $(n+k)$-dimensional Riemannian manifold $(\mc{N},\bar{g})$, and let $\nabla,\bar{\nabla}$ be the respective Levi-Civita connections. We shall write $g(\cdot,\cdot) = \metric{\cdot}{\cdot}$, classically called the \emph{first fundamental form} of $\cM$ in $\mc{N}$. Then, at any point $p\in\cM$ and for any vector $\by\in T_p\cM$, and vector field $\bs{X}\in\Gamma(T\cM)$, the metric connection of $\cM$ is the projection of the metric connection of $\mc{N}$: $\nabla_{\by}\bs{X} = (\bar{\nabla}_{\by}\bs{X})^\top$, where $(\;\cdot\;)^\top:T_p\mc{N}\rightarrow T_p\cM$. The \emph{second fundamental form $\II$ of $\cM$ in $\mc{N}$} is defined to be the normal projection of the ambient covariant derivative when acting on vectors fields tangent to $\cM$, i.e., denoting $(\;\cdot\;)^\perp:T_p\mc{N}\rightarrow N_p\cM$,
\begin{equation}
	\II(\bx,\by) = (\bar{\nabla}_{\by}\bs{X})^\perp, \quad \text{i.e.,} \quad \bar{\nabla}_{\by}\bs{X} = \nabla_{\by}\bs{X} + \II(\bx,\by),
\end{equation}
for all $\bx,\by\in T_p\cM$, and $\bs{X}\in\Gamma(T\cM)$ such that $\bs{X}\vert_p=\bx$. It is a symmetric bilinear form on the tangent space at every point taking values in the normal space, $\II:T_p\cM\otimes T_p\cM\rightarrow N_p\cM$. Fixing a normal vector $\bN\in N_p\cM$, the scalar-valued bilinear form $\metric{\II(\bx,\by)}{\bN}$ has a corresponding self-adjoint map $\oS_{\bN}\in\text{End}(T_p\cM)$, called the \emph{Weingarten map at $\bN$}, such that:
\begin{equation}
	\metric{\II(\bx,\by)}{\bN} = \metric{\oS_{\bN}\,\bx}{\by} = \metric{\bx}{\oS_{\bN}\,\by}.
\end{equation}
Fixing orthonormal bases $\{\bE_\mu\}_{\mu=1}^n$ of $T_p\cM$, and $\{\bN_j\}_{j=1}^k$ of $N_p\cM$, the components of the second fundamental form at point $p$ are:
\begin{equation}
	\II(\bE_\mu,\bE_\nu) = \sum_{j=1}^k \mr{II}^j(\bE_\mu,\bE_\nu)\bN_j = \sum_{j=1}^k \metric{\II(\bE_\mu,\bE_\nu)}{\bN_j}\;\bN_j = \sum_{j=1}^k \metric{\oS_j\,\bE_\mu}{\bE_\nu}\;\bN_j.
\end{equation}
The geometric meaning of $\II$ lies in the fact that the Weingarten map measures the tangential rate of change of normal vectors to $\cM$ when moving in tangent directions, cf. \cite[Eq. II.2.4]{chavel2006}:
$$
	\oS_{\bN}\,\bx = - (\bar{\nabla}_{\bx}\bs{N})^\top,
$$
for any $\bs{N}\in\Gamma(N\cM)$ such that $\bs{N}\vert_p=\bN$. From this, \cite[Ch. 4, Cor. 9, 10]{oneil1983}, $\II(\bx,\bx)$ is to be interpreted as the curve acceleration in $\mc{N}$ of a geodesic inside $\cM$ at $p$ with tangent velocity $\bx$. Therefore, $\II$ naturally measures the extrinsic curvature of the embedding since it represents the forced curving of the straightest lines in $\cM$ due to the curving of $\cM$ itself in $\mc{N}$.

The inverse function theorem and \cite[Ch. VII, Ex. 3.3]{kobayashi1969} establish the following lemma, of fundamental importance for the computations in the proofs of the present work.

\begin{lemma}
Let $\cM$ be an $n$-dimensional submanifold of an $(n+k)$-dimensional Riemannian manifold $(\mc{N},\, g)$, with the induced metric $g\vert_{\cM}$. For any point $p\in\cM$ and orthonormal basis $\{\bE_\mu\}_{\mu=1}^n$ of $T_p\cM$, it is possible to choose normal coordinates $(y^1,\dots,y^{n+k})$ in $\mc{N}$ such that the coordinate tangent vectors at the origin $\bs{Y}^1,\dots,\bs{Y}^n$ coincide with $\{\bE_\mu\}_{\mu=1}^n$, and $\bs{Y}^{n+1},\dots,\bs{Y}^{n+k}$ are an orthonormal basis $\{\bN_j\}_{j=1}^k$ of $N_p\cM$. Moreover, $\cM$ is locally given by a graph manifold $y^1=x^1,\dots, y^n=x^n,y^{n+1}=f^1(\bx),\dots, y^{n+k}=f^k(\bx)$, such that the components of the second fundamental form at $p$ can be written as:
\begin{equation}\label{eq:IIlocal}
	\II(\bE_\mu,\bE_\nu) =\sum_{j=1}^k \left[ \frac{\dd^2 f^j}{\dd x^\mu\dd x^\nu}(0) \right]\,\bN_j .
\end{equation}
\end{lemma}

The invariance of the trace of $\II$ for any orthonormal tangent frame $\{\bE_\mu\}_{\mu=1}^n$ leads to the definition of the \emph{mean curvature} vector:
\begin{equation}
\bH = \sum_{\mu=1}^n\;\II(\bE_\mu,\bE_\mu) = \sum_{j=1}^k H^j\bN_j, \qquad\text{ where } H^j=\sum_{\mu=1}^n\mr{II}^j(\bE_\mu,\bE_\mu).
\end{equation}

The study of the intrinsic geometry of $(\cM,g)$ depends only on the metric and is given in terms of the Riemann curvature tensor:
$$
\bR(\bx,\by)\bz=(\nabla_{\bx}\nabla_{\by}-\nabla_{\by}\nabla_{\bx}-\nabla_{[\bx,\by]})\bs{Z},
$$
for any $\bx,\by,\bz\in T_p\cM$ and $\bs{Z}\in\Gamma(T\cM)$ such that $\bs{Z}\vert_p=\bz$. This fundamental tensor equivalently measures the integrability of parallel transport, geodesic deviation and local flatness. Its traces yield the \emph{Ricci tensor} 
$$
\Ric{\bx}{\by}=\sum_{\mu=1}^n\Rie{\bE_\mu}{\bx}{\by}{\bE_\mu} = \metric{\oR\,\bx}{\by},
$$
and the \emph{scalar curvature}, $\cR = \sum_{\mu}\Ric{\bE_\mu}{\bE_\mu}$. Here, $\oR\in\text{End}(T_p\cM)$ is the \emph{Ricci operator} associated to the Ricci tensor with respect to the metric.

Gau{\ss} Theorema Egregium establishes that the intrinsic curvature of surfaces is a particular combination of products of the components of the second fundamental form. This generalizes to higher dimension to %\cite[Ch. VII, Prop. 4.1]{kobayashi1969},

\begin{theorem}[Gau{\ss} equation]\label{th:Gauss}
	The Riemann curvature tensor of a submanifold $\cM$ is related to the curvature $\bar{\bR}$ of the ambient manifold $\mc{N}$ via
\begin{equation}
	\Rie{\bx}{\by}{\bz}{\bw} = \metric{\bar{\bR}(\bx,\by)\bz}{\bw} +\metric{\II(\bx,\bw)}{\II(\by,\bz)} - \metric{\II(\bx,\bz)}{\II(\by,\bw)}
\end{equation}
for all $\bx,\by,\bz,\bw\in T_p\cM$.
\end{theorem}

In classical differential geometry, \cite{gray2006}, \cite{toponogov2006}, the third fundamental form is the natural object to construct out of scalar products after the first fundamental form, $\text{I}(\bx,\by)=\langle\bx,\by\rangle$, and the second fundamental form $\text{II}(\bx,\by)=\langle\oS\,\bx,\by\rangle$, so it is defined for hypersurfaces, e.g. \cite{liu1999}, as $$\text{III}(\bx,\by)=\metric{\oS\,\bx}{\oS\,\by} =\langle \oS^2\bx,\by\rangle.$$ However, it does not provide new information since it is completely determined by Gau{\ss} equation \ref{th:Gauss}, e.g., in Euclidean space \cite{kobayashi1969}:
\begin{equation}\label{eq:R-Wein}
\metric{\bs{\hat{S}}^2\bx}{\by}=H\metric{\oS\,\bx}{\by} - \Ric{\bx}{\by},	
\end{equation}
or, in terms of the Ricci operator, $\oS^2 = H\oS - \oR$. For a manifold $\cM$ of higher codimension $k$, there
are $k$ linearly independent normal vectors at every point and, as mentioned before, the generalized second fundamental form takes values in the normal bundle precisely to reflect this structure in terms of the corresponding Weingarten operators at every normal vector. Therefore, the natural generalization of $\metric{\oS\,\bx}{\oS\,\by}$ to this context is

\begin{definition}\label{def:III}
	The \emph{third fundamental form} of a Riemannian submanifold $\cM\subset\mc{N}$ is the fourth-rank tensor $\III\in (T_p\cM^*)^2\otimes N_p\cM^*\otimes N_p\cM$, given at every point $p\in\cM$ by
\begin{equation}
	\langle\; \III(\bx,\by)\,\bN,\,\bM\,\rangle := \langle\,\oS_{\bM}\;\bx\, ,\,\oS_{\bN}\;\by\,\rangle.
\end{equation}
for any $\bx,\by\in T_p\cM$, and $\bN,\bM\in N_p\cM$.
\end{definition}

At any specific point, and because the Weingarten maps are self-adjoint, the linear operator $\III(\bx,\by)\in\text{End}(N_p\cM)$ is written as the following linear combination, when a particular orthonormal basis $\{\bN_j\}_{j=1}^k$ of the normal space is fixed and $\bs{\eta}^j=g(\cdot,\bN_j)$ is the dual basis:
\begin{equation}
	\III(\bx,\by) = \sum_{i,\, j=1}^k \langle\,\oS_i\,\oS_j\;\bx\, ,\,\by\,\rangle\;\bs{\eta}^i\otimes\bN_j.
\end{equation}
This is due to the linearlity of the map $\bN\mapsto\oS_{\bN}:N_p\cM\rightarrow\text{End}(T_p\cM)$; if $\bN=\sum_j n^j\bN_j$ then
$$
\langle\,\oS_{\bN}\;\bx ,\,\by\,\rangle = \langle\,\II(\bx,\by),\,\bN\,\rangle = \sum_{j=1}^k n^j\langle\,\II(\bx,\by),\,\bN_j\rangle =\langle\,\left(\sum_{j=1}^k n^j\oS_j\right)\bx,\,\by\,\rangle,
$$ 
for all $\bx,\by\in T_p\cM$.

Let us define the \emph{tangent trace} of a tensor $\bs{A}\in (T_p\cM^*)^2\otimes N_p\cM^*\otimes N_p\cM$ as the operator sum of the evaluations at an orthonormal basis $\{\bE_\mu\}_{\mu=1}^n$ of $T_p\cM$:
\begin{equation}
	\tr_\parallel\bs{A} := \sum_{\mu=1}^n \bs{A}(\bE_\mu,\bE_\mu) \; \in\text{End}(N_p\cM),
\end{equation}
And let the \emph{normal trace} of such a tensor be 
\begin{equation}
	\tr_\perp\bs{A} := \sum_{j=1}^k \langle\; \III(\cdot,\cdot)\,\bN_j,\,\bN_j\,\rangle\;\in (T_p\cM^*)^2,
\end{equation}
for any orthonormal basis $\{\bN_j\}_{j=1}^k$ of $N_p\cM$. These tensors are well-defined since the sums are independent of the orthonormal basis chosen.

\begin{lemma}\label{lem:trnIII}
At any point $p\in\cM$, for any $\bx,\by\in T_p\cM$, and $\bN,\bM\in N_p\cM$, the normal trace of the third fundamental form is
\begin{equation}
	\tr_\perp\III (\bx,\by) =  \sum_{j=1}^k\langle\,\oS_j^2\;\bx,\,\by \,\rangle = \langle\, (\SH-\oR+\bar{\bs{\mc{R}}})\,\bx,\,\by \,\rangle ,
\end{equation}
where $\oR$ and $\bar{\bs{\mc{R}}}$ are the Ricci operators of $\cM$ and $\mc{N}$ respectively. In particular, the sum of squares of the Weingarten operators $\oS_j$, for an orthonormal basis $\{\bN_j\}_{j=1}^k$ of $N_p\cM$, is independent of the basis. The tangent trace of the third fundamental form is a linear operator on $N_p\cM$ whose components with respect to the metric are the Frobenius inner products of the corresponding Weingarten operators:
\begin{equation}
	\langle\; (\tr_\parallel\III)\;\bN,\,\bM\,\rangle = \tr(\oS_{\bN}\oS_{\bM}).
\end{equation}
The total trace is
\begin{equation}\label{eq:trIII}
	\tr\III = \tr_\perp\tr_\parallel\III = \|\bH\|^2 -\cR +\bar{\cR}.
\end{equation}
\end{lemma}
\begin{proof}
The normal trace biliniar form has components
\begin{align}
\tr_\perp\III(\bE_\mu,\bE_\nu) & = \sum_{j=1}^k\metric{\oS_j\;\bE_\mu}{\oS_j\;\bE_\nu}= \sum_{j=1}^k\sum_{\al=1}^n\metric{\oS_j\bE_\al}{\bE_\mu}\metric{\oS_j\bE_\al}{\bE_\nu}\nonumber \\
& = \sum_{j=1}^k\sum_{\al=1}^n \mr{II}^j(\bE_\al,\bE_\mu)\mr{II}^j(\bE_\al,\bE_\nu) = \sum_{\al=1}^n\metric{\II(\bE_\al,\bE_\mu)}{\II(\bE_\al,\bE_\nu)}, \label{eq:trNcomp}
\end{align}
that using Gau{\ss} equation lead to the corresponding linear operator with respect to the metric:
\begin{align*}
\tr_\perp\III(\bE_\mu,\bE_\nu) & = \sum_{\al=1}^n\metric{\II(\bE_\al,\bE_\al)}{\II(\bE_\nu,\bE_\mu)} + \sum_{\al=1}^n \metric{\bar{\bR}(\bE_\al,\bE_\nu)\bE_\mu}{\bE_\al} - \sum_{\al=1}^n \metric{\bR(\bE_\al,\bE_\nu)\bE_\mu}{\bE_\al} \\
& = \metric{\II(\bE_\mu,\bE_\nu)}{\bH} + \bar{\bs{\mc{R}ic}}(\bE_\mu,\bE_\nu) - \Ric{\bE_\mu}{\bE_\nu} \\
& = \metric{\oS_{\bH}\,\bE_\mu}{\bE_\nu}+\metric{\bs{\bar{\mc{R}}}\,\bE_\mu}{\bE_\nu}-\metric{\bs{\hat{\mc{R}}}\,\bE_\mu}{\bE_\nu}.
\end{align*}
This is the generalization of the operator of the classical third fundamental form, equation \ref{eq:R-Wein}: $$\displaystyle \sum_{j=1}^k \oS_j^2 = \SH-\oR+\bar{\bs{\mc{R}}}.$$
The tangent trace is trivial by definition of trace of a linear operator with respect to the metric and the self-adjointness of the Weingarten operators:
$$
\metric{(\tr_\parallel\III)\;\bN}{\bM} = \sum_{\mu=1}^n\metric{\oS_{\bM}\oS_{\bN}\,\bE_\mu}{\bE_\mu} = (\oS_{\bM},\oS_{\bN})_F.
$$
In a fixed orthonormal basis this tensor is the linear combination
$$
\tr_\parallel\III = \sum_{i,\, j=1}^k \sum_{\mu=1}^n\langle\,\oS_i\,\oS_j\;\bE_\mu\, ,\,\bE_\mu\,\rangle\;\bs{\eta}^i\otimes\bN_j =  \sum_{i,\, j=1}^k \tr(\oS_i\,\oS_j) \;\bs{\eta}^i\otimes\bN_j,	
$$
whose components can be expressed in terms of the second fundamental form as
\begin{equation}\label{eq:trTcomp}
\tr(\oS_i\oS_j) =\sum^n_{\mu,\,\nu=1}\metric{\oS_i\;\bE_\mu}{\bE_\nu}\metric{\oS_j\;\bE_\mu}{\bE_\nu} =  \sum^n_{\mu,\,\nu=1} \mr{II}^i(\bE_\mu,\bE_\nu)\mr{II}^j(\bE_\mu,\bE_\nu).
\end{equation}
Taking the total trace of $\III$ is analogous to the complete contraction of the Riemann curvature tensor indices to obtain the scalar curvature:
\begin{align}
\tr\III = \tr_\parallel\tr_\perp\III & = \sum_{\mu=1}^n \metric{(\SH-\oR+\bar{\bs{\mc{R}}})\bE_\mu}{\bE_\mu} = \tr\SH - \tr\oR+\tr\bar{\bs{\mc{R}}}\nonumber \\
& = \sum_{\mu=1}^n\tr_\perp\III(\bE_\mu,\bE_\mu) = \sum_{\al,\,\bet}^n\|\II(\bE_\al,\bE_\bet)\|^2,\label{eq:trTotal}
\end{align}
where $\tr\SH =\sum_{\mu=1}^n\metric{\II(\bE_\mu,\bE_\mu)}{\bH} =\|\bH\|^2$, and the traces of the Ricci operators are by definition the scalar curvatures.
\end{proof}

Equations \ref{eq:trTcomp} and \ref{eq:trNcomp} will be recognized inside the elements of the tangent and normal matrix blocks in our covariance matrices to express its eigenvalues in terms of the third fundamental form.

The asymmetry of the components of the third fundamental form operator $\III(\bx,\by)$ encodes the curvature information of the connection defined on the normal bundle $N\cM$ by $(\bar{\nabla}_{\bx}\bs{N})^\perp$, for any $\bx\in T_p\cM,\;\bs{N}\in \Gamma(N\cM)$, where an analog to Gau{\ss} equation holds.

\begin{lemma}[Ricci equation]
	The Riemann curvature of the induced normal connection, $\bR_\perp$, satisfies:
\begin{equation}
	\langle\bR_\perp(\bx,\by)\bN,\bM\,\rangle =\langle\,\bar{\bR}(\bx,\by)\bN,\bM\,\rangle + \langle\,\III(\bx,\by)\bN,\bM\,\rangle - \langle\,\III(\bx,\by)\bM,\bN\,\rangle ,
\end{equation}
for all $\bx,\by\in T_p\cM$, and $\bN,\bM\in N_p\cM$, at any point $p\in\cM$.
\end{lemma}
\begin{proof}
Writing the classical equation \cite[Ex. II.11]{chavel2006} in terms of Weingarten maps leads to
\begin{align*}
& \langle\,\bar{\bR}(\bx,\by)\bN,\bM\,\rangle -\langle\bR_\perp(\bx,\by)\bN,\bM\,\rangle =\sum_{\mu=1}^n\left[\; \metric{\II(\bE_\mu,\bx)}{\bN}\metric{\II(\bE_\mu,\by)}{\bM} +\right. \\ & \qquad\qquad\qquad\qquad\qquad\qquad\qquad\qquad\qquad\qquad\qquad \left. - \metric{\II(\bE_\mu,\by)}{\bN}\metric{\II(\bE_\mu,\bx)}{\bM}\; \right] \\
& =\sum_{\mu=1}^n \metric{\oS_{\bN}\;\bx}{\bE_\mu} \metric{\oS_{\bM}\;\by}{\bE_\mu} - \metric{\oS_{\bN}\;\by}{\bE_\mu} \metric{\oS_{\bM}\;\bx}{\bE_\mu} = \metric{\oS_{\bN}\;\bx}{\oS_{\bM}\;\by} - \metric{\oS_{\bN}\;\by}{\oS_{\bM}\;\bx}.
\end{align*}
for any orthonormal basis $\{\bE_\mu\}_{\mu=1}^n$ of $T_p\cM$.
\end{proof}

%%%%%%%%%%%%%%%%%%%%%%%%%%%%%%%%%%%%%%%%%%%%%%%%%%%%%%%%%%%%%%%%%%%%%%
%%%%%%%%%%%%%%%%%%%%%%%%%%%%%%%%%%%%%%%%%%%%%%%%%%%%%%%%%%%%%%%%%%%%%%
%%%%%%%%%%%%%%%			3. CYLINDRICAL COVARIANCE
%%%%%%%%%%%%%%%%%%%%%%%%%%%%%%%%%%%%%%%%%%%%%%%%%%%%%%%%%%%%%%%%%%%%%%
%%%%%%%%%%%%%%%%%%%%%%%%%%%%%%%%%%%%%%%%%%%%%%%%%%%%%%%%%%%%%%%%%%%%%%

\section{Cylindrical Covariance Analysis}\label{sec:cylCov}

In this section we compute the integral invariants of the cylindrical domain around a point on an $n$-dimensional submanifold $\cM$ of $\RR^{n+k}$. In the case the cylinder is not normal to the manifold at the point, we can only establish the leading order terms, but that is sufficient in the generic case to be able to detect the tangent space of the manifold by the scaling behaviour of the eigenvalues of the covariance matrix. Once the cylinder is fixed to be normal to this tangent space, the integral invariants can be computed to next-to-leading order to see how they encode the geometric information of the third fundamental form.

We shall always work in a neighborhood $U\subset\RR^{n+k}$ of $p\in\cM$, sufficiently small so that $U\cap\cM$ is given by a graph representation $[x^1,\dots,x^n, f^1(\bx),\dots,f^k(\bx)]^T$ over its tangent space, i.e., $\bze$ represents $p$, $\bx=[x^1,\dots,x^n]^T\in T_p\cM$, and $\nabla f^j(\bze)=\bze$, so that the manifold is approximated at $p$ by its osculating paraboloids.

\begin{lemma}
	The first fundamental form components of a graph manifold $\cM\subset\RR^{n+k}$, parametrized by $[x^1,\dots, x^n,f^1(\bx),\dots,f^k(\bx)]^T\in T_p\cM\oplus N_p\cM\cong\RR^{n+k}$, are:
\begin{equation}
	 g_{\mu\nu}(\bx) = \del_{\mu\nu} + \sum_{j=1}^k \frac{\dd f^j}{\dd x^\mu}\frac{\dd f^j}{\dd x^\nu} .
\end{equation}
The induced measure on $\cM$ in these coordinates is given by
\begin{equation}
\mr{dVol}= \sqrt{\det g(\bx)}\,d^n\bx = \left(1 + \frac{1}{2}\sum_{j=1}^k\sum^n_{\al=1}\left[\sum_{\bet=1}^n\left(\frac{\dd^2 f^j}{\dd x^\al\dd x^\bet}(0)\right) x^\bet \right]^2 +\cO(x^3)\right)\, d^n\bx.
\end{equation}
\end{lemma}
\begin{proof}
	The tangent space in these coordinates is spanned by the vectors
	$$
	\bs{X}_\mu = \frac{\dd}{\dd x^\mu}[x^1,\dots, x^n,f^1(\bx),\dots,f^k(\bx)]^T = [0,\dots,1,\dots, 0,\frac{\dd f^1}{\dd x^\mu},\dots,\frac{\dd f^k}{\dd x^\mu}]^T,
	$$
	for $\mu=1,\dots, n$, which yields the canonical orthonormal basis at $p$ since $\nabla f^j(\bze)=\bze$. The induced metric tensor is then
	$$
	g_{\mu\nu}(\bx) = \metric{\bs{X}_\mu}{\bs{X}_\nu} = \del_{\mu\nu} + \sum_{j=1}^k \frac{\dd f^j}{\dd x^\mu}\frac{\dd f^j}{\dd x^\nu}.
	$$
	From this, recalling that the $f^j(\bx)$ have Taylor expansions starting at order $2$ in these coordinates, the matrix of the metric components is of the form $[g]=\Id_n + [h]$, where the correction matrix $[h]=[\sum_j\dd_\mu f^j\dd_\nu f^j]$ is small because we are in a neighborhood of $\bze$ with $\nabla f^j(\bze)=\bze$. Let
	$$
	f^j(\bx) = \frac{1}{2}\sum_{\al,\bet =1}^n \left(\frac{\dd^2 f^j}{\dd x^\al\dd x^\bet}(0)\right) x^\al x^\bet + \cO(x^3),
	$$
	 for every $j=1,\dots, k$, then
	$$
	\frac{\dd f^j}{\dd x^\mu}  =\sum_{\bet=1}^n\left(\frac{\dd^2 f^j}{\dd x^\bet\dd x^\mu}(0)\right) x^\bet  +\cO(x^2).
	$$
%	$$
%	\frac{\dd f^j}{\dd x^\mu} =  \frac{1}{2}\sum_{\al,\bet =1}^n \left(\frac{\dd^2 f^j}{\dd x^\al\dd x^\bet}(0)\right) (\del_{\al\mu}x^\bet + x^\al\del_{\bet\mu}) + \cO(x^2) =\sum_{\bet=1}^n\left(\frac{\dd^2 f^j}{\dd x^\bet\dd x^\mu}(0)\right) x^\bet  +\cO(x^2).
%	$$
The natural volume form of a Riemannian manifold is given by $\sqrt{\det g}\; dx^1\wedge\dots\wedge dx^n$, \cite[Ch. 7, Lem. 19]{oneil1983}, whose lowest order approximation is $\det g\approx 1 + \tr h$, so $\sqrt{\det g}\approx 1+\frac{1}{2}\tr h$, i.e.,
	$$
	\sqrt{\det g(\bx)} = 1 + \frac{1}{2}\sum_{\al=1}^n\sum_{j=1}^k\left(\frac{\dd f^j}{\dd x^\al}\right)^2+... = 1 + \frac{1}{2}\sum_{\al=1}^n\sum_{j=1}^k\left[ \sum_{\bet=1}^n\left(\frac{\dd^2 f^j}{\dd x^\bet\dd x^\mu}(0)\right) x^\bet  \right]^2 + \cO(x^3).
	$$
\end{proof}

 In the rest of this paper we shall abbreviate second derivatives at the origin by
$$
\ka^j_{\al\bet}= \ka^j_{\bet\al} :=\frac{\dd^2 f^j}{\dd x^\al\dd x^\bet}(0) ,
$$
motivated by the notation of hypersurface principal curvatures, which are the eigenvalues of the local Hessian of the defining function.
We can now compute the Taylor expansion of the integral invariants in the chosen coordinates, and then relate the terms to the curvature differential invariants which are always combinations of second derivatives.

\begin{theorem}\label{th:volCyl}
	The $n$-dimensional volume of the cylindrical component for a generic $\VV\in\Gr{n}{n+k}$, such that $\VV^\perp\cap T_p\cM=\{\bze\}$, is to leading order the volume of the ellipsoid of intersection between the $\VV$-cylinder and $T_p\cM$:
\begin{equation}
	V(\Cyll(\vep,\VV)) = V_n(1)\prod_{\mu=1}^n \ell_\mu +\cO(\vep^{n+1}),
\end{equation}
where $\ell_\mu$ are the the \emph{principal semi-axes} of the ellipsoid. When $\VV=T_p\cM$, the volume is
\begin{equation}\label{eq:volCyl}
	V(\Cyll(\vep)) = V_n(\vep)\left[ 1+\frac{\vep^2}{2(n+2)}\;\tr\III +\cO(\vep^4) \right]
\end{equation}
where $\tr\III=\|\bH\|^2-\cR$.
\end{theorem}
\begin{proof}
To compute the leading term of $V(\Cyll(\vep,\VV))$ we can approximate $\cM$ near $p$ by its tangent space, such that, fixing local coordinates with a basis for $T_p\cM\oplus N_p\cM$, a point is specified by $\bs{X}=[\bx,\bze]^T$, with $\bx\in T_p\cM,\;\bze\in N_p\cM$. Since $\VV^\perp\cap T_p\cM=\{\bze\}$, we have $T_p\cM\oplus\VV^\perp=\RR^{n+k}$, and of course $\VV\oplus\VV^\perp=\RR^{n+k}$. Let $\{\bE_\mu\}_{\mu=1}^n$ be an orthornomal basis of $T_p\cM$, and $\{\bu_\al\}_{\al=1}^n\cup\{\bv_j\}_{j=1}^k$ an orthonormal basis of $\VV\oplus\VV^\perp$, then the elements of the former are a linear combination of the latter, so there are matrices $A, B$ such that:
$$
\bE_\mu = \sum_{\al=1}^n A^\al_\mu \bu_\al + \sum_{j=1}^k B^j_\mu\bv_j .
$$
We need to find the region $\|\projj{\bs{X}}{\VV}\|\leq\vep$, and since $\bs{X}=\sum_{\mu}x^\mu\bE_\mu$, when $\bs{X}\in T_p\cM$, the projection is
$$
\projj{\bs{X}}{\VV} = \sum_{\al=1}^n \metric{\bs{X}}{\bu_\al}\;\bu_\al = \sum_{\al=1}^n\sum_{\mu=1}^n x^\mu A^\al_\mu\bu_\al,
$$
hence, the domain of integration in $\bx$ in this approximation is
$$
\|\projj{\bs{X}}{\VV}\|^2 =\sum_{\al=1}^n\left(\sum_{\mu=1}^n x^\mu A^\al_\mu\right)^2  \leq\vep^2.
$$
This is a quadratic equation that can be written as
\begin{align*}
	\sum_{\mu ,\,\nu}^n x^\mu\left[ \sum_{\al=1}^n A^\al_\mu A^\al_\nu \right] x^\nu = \bx^T [A\cdot A^T] \bx = \by^T\cdot\by=\|\by\|^2 \leq \vep^2,
\end{align*}
where $\by=A^T\bx$. The matrix $[A\cdot A^T]$ is positive definite since it is clearly nonnegative, and if $\bx\in\ker A^T$ for nonzero $\bx$, then $\projj{\bs{X}}{\VV}=\bze$, thus $\bs{X}\in\VV^\perp$, which contradicts $\bs{X}\in T_p\cM$ under our assumption $\VV^\perp\cap T_p\cM=\{\bze\}$. Therefore, the cylindrical domain is an $n$-dimensional ellipsoid in the tangent space at $p$, whose volume is given in terms of its principal semi-axes:
$$
V(\Cyll(\vep,\VV)) =\frac{\pi^{n/2}}{\Gamma(\frac{n}{2}+1)}\prod^n_{\mu=1}\ell_\mu + \cO(\vep^{n+1}) .
$$

When $\VV=T_p\cM$, the local graph approximation of $\cM$ over $T_p\cM$ yields
$$\projj{\bs{X}}{T_p\cM} = \|\projj{[\bx,f^1(\bx),\dots,f^k(\bx)]^T}{T_p\cM}\|=\|\bx\| \leq\vep,$$
thus, we are integrating $\sqrt{\det g(\bx)}$ over the ball $B^{(n)}_p(\vep)\subset T_p\cM$, which can be computed using the integrals in the appendix:
\begin{align*}
	V(\Cyll(\vep)) & = \int_{\SS^{n-1}}d\,\SS\int^\vep_0 \rho^{n-1}\left( 1+\frac{1}{2}\sum^k_{i=1}\sum^n_{\al=1}\left[\sum^n_{\bet=1}\ka^i_{\al\bet}\rho\,\cx^\bet \right]^2 +\cO(x^3) \right)d\rho \\
	& = V_n(\vep) + \frac{\vep^{n+2}}{2(n+2)}\sum^k_{i=1}\sum^n_{\al=1}\sum^n_{\bet,\ga}\ka^i_{\al\bet}\ka^i_{\al\ga}\int_{\SS^{n-1}}\cx^\bet\cx^\ga \,d\SS +\cO(\vep^{n+4}) \\
	& = V_n(\vep) + \frac{C_2\,\vep^{n+2}}{2(n+2)}\sum^k_{i=1}\sum^n_{\al,\bet}\;(\ka^i_{\al\bet})^2 +\cO(\vep^{n+4}) \\
	& = V_n(\vep) + \frac{V_n(\vep)\,\vep^{2}}{2(n+2)}\sum^n_{\al,\bet}\;\metric{\II(\bE_\al,\bE_\bet)}{\II(\bE_\al,\bE_\bet)}+\cO(\vep^{n+4}).
\end{align*}
Here the spherical integral is only nonzero when $\bet=\ga$, and the last term is the component expression of equation \ref{eq:trIII}.
\end{proof}

\begin{proposition}
	The barycenter of the cylindrical component, for $\VV$ as in the previous theorem, is 
\begin{equation}
	\bb(\Cyll(\vep,\VV)) = \bze +\cO(\vep^2).
\end{equation}
	In the case $\VV=T_p\cM$, the barycenter is:
\begin{equation}
	\bb(\Cyll(\vep)) = [ \;\bze,\; \frac{\vep^2}{2(n+2)}\;\bH\;]^T +\cO(\vep^4).
\end{equation}
\end{proposition}
\begin{proof}
For generic $\VV$, approximating the manifold again by its tangent space, $\bs{X}=[\bx,\bze +\cO(\vep^2)]^T$, the normal component does not contribute until order two and the tangent component also vanishes at order $1$ in $\vep$. When $\VV=T_p\cM$, we saw that the integration domain reduces to a ball. The integrals of the tangent components $x^\mu$ weighed by $\sqrt{\det g}$ are of order $\cO(\vep^{n+4})$, since the first terms in the expansion have odd powers in the coordinates. On the other hand the normal components integrate as:
\begin{align*}
	V[\bb]^j & = \int_{\SS^{n-1}}d\,\SS\int^\vep_0 f^j\sqrt{\det g}\rho^{n-1}\, d\rho =  \int_{\SS^{n-1}}d\,\SS\int^\vep_0\rho^{n-1}\left(\frac{1}{2}\sum_{\al,\,\bet=1}^n\ka^j_{\al\,\bet}\rho^2\cx^\al\cx^\bet + \cO(x^3)\right) d\rho \\
	& = \frac{\vep^{n+2}}{2(n+2)}\sum_{\al,\bet=1}^n\ka^j_{\al\bet}\int_{\SS^{n-1}}\cx^\al\cx^\bet d\,\SS +\cO(\vep^{n+4}) = \frac{C_2\,\vep^{n+2}}{2(n+2)}\, H^j + \cO(\vep^{n+4}),
\end{align*}
Dividing by $V=V(\Cyll(\vep))$ cancels $C_2\vep^n=V_n(\vep)$ to leading order.
\end{proof}

In order to study the eigenvalue decomposition of the covariance matrix we need to establish how to determine the limit eigenvectors and the first two terms of the series expansion of the eigenvalues, so that computing the integrals in an arbitrary orthonormal basis produces blocks identifiable in terms of the coordinate expressions of the second and third fundamental forms in that basis. An analogous result to the matrix expansion in \cite{alvarez2018a} generalizes to higher codimension.

\begin{lemma}\label{lem:matrix}
Let $C(\vep)$ be an $(n+k)\times(n+k)$ real symmetric matrix depending on a real parameter $\vep$ with convergent series expansion in a neighborhood of $0$ such that:
$$
C(\vep) = \vep^2
\left( \begin{array}{@{}c|c@{}}
   \begin{matrix}
		a\,\Id_{n}
   \end{matrix} 
      & 0_{n\times k} \\
   \cmidrule[0.4pt]{1-2}
   0_{k\times n} & 0_{k\times k}  \\
\end{array}  \right)
	+\,\vep^4
\left( \begin{array}{@{}c|c@{}}
   \begin{matrix}
       \mr{A}_{n\times n}
   \end{matrix} 
      & \mr{B}_{n\times k} \\
   \cmidrule[0.4pt]{1-2}
   \mr{B}_{k\times n} & \Gamma_{k\times k}  \\
\end{array} \right) + \cO(\vep^{5}),
$$
where $a\neq 0$, and the blocks $\mr{A,B},\Gamma$ are not completely zero. Let $[\bV]_\top, [\bV]_\perp$ denote the first $n$ and last $k$ components of a vector in $\RR^{n+k}$. Then the series of eigenvectors of $C(\vep)$ form an orthonormal basis of $\RR^{n+k}$ that converges for $\vep\rightarrow 0$. The first $n$ eigenvalues are $\lbd_\mu(\vep)=a\vep^2+\lbd^{(4)}_\mu\vep^4+\cO(\vep^5)$, where $\lbd^{(4)}_\mu$ and the corresponding limit eigenvectors $\{\bV^{(0)}_\mu\}_{\mu=1}^n$ satisfy the eigenvalue decomposition of $\mr{A}$:
$$
(\lbd^{(4)}_\mu\,\Id_n -\mr{A})\, [\bs{V}^{(0)}_\mu]_\top = 0_{n\times 1},\qquad [\bs{V}^{(0)}_\mu]_\perp = 0_{k\times 1}. 
$$
The last $k$ eigenvalues are $\lbd_j(\vep)=\lbd^{(4)}_j\vep^4+\cO(\vep^5)$, where $\lbd^{(4)}_j$ and the corresponding limit eigenvectors $\{\bV^{(0)}_j\}_{j=n+1}^{n+k}$ satisfy the eigenvalue decomposition of $\Gamma$:
$$
(\lbd^{(4)}_j\,\Id_k -\Gamma)\, [\bs{V}^{(0)}_j]_\perp = 0_{n\times 1},\qquad [\bs{V}^{(0)}_j]_\top = 0_{n\times 1}. 
$$
Therefore, the fourth-order term of the eigenvalues is given by the eigenvalues of the blocks $A$ and $\Gamma$, with the respective eigenvectors as the limit eigenvectors of $C(\vep)$ for $\vep\rightarrow 0$.
\end{lemma}
\begin{proof}
The eigenvalue decomposition $C(\vep)\bV(\vep)=\lbd(\vep)\bV(\vep)$ can be written as a convergent series expansion in $\vep$ within a neighborhood of $0$ for all Hermitian matrices of converging power series elements \cite{rellich1969}:
\begin{align*}
& [\; \vep^2
\left( \begin{array}{@{}c|c@{}}
   \begin{matrix}
		a\,\Id_{n}
   \end{matrix} 
      & 0_{n\times k} \\
   \cmidrule[0.4pt]{1-2}
   0_{k\times n} & 0_{k\times k}  \\
\end{array}  \right)
	+\vep^4
\left( \begin{array}{@{}c|c@{}}
   \begin{matrix}
       \mr{A}_{n\times n}
   \end{matrix} 
      & \mr{B}_{n\times k} \\
   \cmidrule[0.4pt]{1-2}
   \mr{B}_{k\times n} & \Gamma_{k\times k}  \\
\end{array} \right) + \cO(\vep^{5})\; ]\cdot [\, \bs{V}^{(0)} + \bs{V}^{(1)}\vep+\bs{V}^{(2)}\vep^2+\dots \,] = \\
&\qquad = (\lbd^{(1)}\vep^1+\lbd^{(2)}\vep^2+\lbd^{(3)}\vep^3+\lbd^{(4)}\vep^4+\dots)[\, \bs{V}^{(0)} + \bs{V}^{(1)}\vep+\bs{V}^{(2)}\vep^2+\dots \,].
\end{align*}
The zero matrix $C(0)$ is the limit when $\vep\rightarrow 0$, with $\lbd(0)=\lbd^{(0)}=0$ as a totally degenerate eigenvalue of multiplicity $(n+k)$. By \cite[ch. I, Th. 1]{rellich1969}, for $\vep >0$, this eigenvalue branches out  into $(n+k)$ eigenvalues $\lbd_i(\vep)$ with $(n+k)$ orthonormal eigenvectors $\bs{V}_i(\vep)$, all convergent in a neighborhood of $0$. Thus, the vectors $\bV^{(0)}_i=\lim_{\vep\rightarrow 0}\bV_i(\vep)$ are a unique orthonormal basis of $\RR^{n+k}$ that is completely determined by the perturbation matrix.

The eigenvalue difference between $C(\vep)$ and its full diagonalization is bounded by the matrix norm difference between them, which implies $\lbd^{(1)}=\lbd^{(3)}=0$, and also $\lbd^{(2)}_i=a$, for $i=1,\dots,n$, and $\lbd^{(2)}_i=0$, for $i=n+1,\dots,n+k$, since $C(\vep)$ is already diagonal up to that order. One can obtain the relations satisfied by $\lbd^{(4)}$ and $\bV^{(0)}$ equating order by order. 
At second order, $\lbd^{(2)}_i=a$ is nonzero for $i=1,\dots,n$, hence
\begin{align*}
& [\; 
\left( \begin{array}{@{}c|c@{}}
   \begin{matrix}
		a\,\Id_{n}
   \end{matrix} 
      & 0_{n\times k} \\
   \cmidrule[0.4pt]{1-2}
   0_{k\times n} & 0_{k\times k} \\
\end{array}  \right)
	-\lbd^{(2)}_i\,\Id_{n+k}\, ]\,\bs{V}^{(0)}_i =
\left( \begin{array}{@{}c|c@{}}
   \begin{matrix}
		0_{n\times n}
   \end{matrix} 
      & 0_{n\times k} \\
   \cmidrule[0.4pt]{1-2}
   0_{k\times n} & -a\,\Id_k  \\
\end{array}  \right)\,\bs{V}^{(0)}_i = 0
\end{align*}
implies that $[\bV^{(0)}_\mu]_\perp=0_{k\times 1}$, for the limit of the first $n$ eigenvectors. At fourth order we have
\begin{align*}
&
[\, \lbd^{(4)}_i\,\Id_{n+k}- 
\left( \begin{array}{@{}c|c@{}}
   \begin{matrix}
		\mr{A}_{n\times n}
   \end{matrix} 
      & \mr{B}_{n\times k} \\
   \cmidrule[0.4pt]{1-2}
   \mr{B}_{k\times n} & \Gamma_{k\times k}  \\
\end{array}  \right) 
	 \, ]\,\bs{V}^{(0)}_i
=
[\; 
\left( \begin{array}{@{}c|c@{}}
   \begin{matrix}
		a\,\Id_{n}
   \end{matrix} 
      & 0_{n\times k} \\
   \cmidrule[0.4pt]{1-2}
   0_{k\times n} & 0_{k\times k}  \\
\end{array}  \right)
	-\lbd^{(2)}_i\,\Id_{n+k}\, ]\,\bs{V}^{(2)}_i,
\end{align*}
which in the present case, $i=1,\dots,n$, makes the right-hand side become $0$ for the first $n$ rows. On the other hand, $[\bV^{(0)}_i]_\perp=0_{k\times 1}$ makes $\mr{B}$ not contribute in the left-hand side, hence the first $n$ rows lead to the equation:
$$
(\lbd^{(4)}_i\,\Id_n -\mr{A})\,[\bs{V}^{(0)}_i]_\top = 0_{n\times 1}.
$$
When $i=n+1,\dots,n+k$, an analogous argument using $\lbd^{(2)}_i=0$, leads to $[\bV^{(0)}_i]_\top=0_{n\times 1}$, and in turn to:
$$
(\lbd^{(4)}_i\,\Id_n -\Gamma)\,[\bs{V}^{(0)}_i]_\perp = 0_{k\times 1}.
$$
Since the limit eigenvectors are an orthonormal basis they cannot be zero and, therefore, the previous equations establish $\lbd^{(4)}_i$ and the nonzero components of $[\bV^{(0)}_i]$ as the eigenvalue decomposition of $\mr{A}$ and $\Gamma$, which always has a solution due to being symmetric matrices.
\end{proof}

The previous lemma is a fundamental step to establish the main theorem of this and the next section.

\begin{theorem}\label{MainTh}
	 For $\VV\in\Gr{n}{n+k}$ such that $\VV^\perp\cap T_p\cM=\{\bze\}$, i.e. for non-normal transversality, and when $\Cyll(\vep,\VV)$ is finite, the covariance matrix $C_p(\vep,\VV)$ has as limit eigenvectors spanning $T_p\cM$ those corresponding to the first $n$ eigenvalues, which scale as $\vep^2$. The other $k$ eigenvalues scaling at higher order have limit eigenvectors that span $N_p\cM$:
		\begin{align}
			& \lbd_\mu (\Cyll(\vep,\VV)) =\frac{\vep^2}{n+2}\ell^2_\mu\, V_n(1)\prod_{\al=1}^n \ell_\al+\;\cO(\vep^{n+3}),\quad\quad\;\; \mu=1,\dots, n, \\
			& \lbd_j (\Cyll(\vep,\VV)) = 0+\cO(\vep^{n+3}),\qquad\quad\qquad\quad\qquad\quad\qquad\; j=n+1,\dots, n+k,
		\end{align}
		where $\ell_\mu$ are the \emph{principal lengths} of the ellipsoid in \ref{th:volCyl}. When $\VV=T_p\cM$, let $\lbd_l[\cdot]$ denote taking the $l$-th eigenvalue of a linear operator at $p$, or of its associated bilinar form with respect to the metric. Then the eigenvalues of the covariance matrix of the cylindrical component are:
\begin{align}\label{eq:tanEVDCyl}
		&  \lbd_\mu (\Cyll(\vep)) =V_n(\vep)\left[\frac{\vep^2}{n+2}+\frac{\vep^4}{2(n+2)(n+4)}\lbd_\mu[\,(\tr\III )\,\Id_n +2\,\tr_\perp\III\,] +\;\cO(\vep^6)\right] \\
& \lbd_j (\Cyll(\vep)) =V_n(\vep)\left[\frac{\vep^4}{4(n+2)(n+4)} \lbd_j[\,\bH\otimes\bH + 2\,\tr_\parallel\III\,] +\cO(\vep^6)\right] \label{eq:norEVDCyl}
\end{align}
for all $\mu=1,\dots, n$, and $j=n+1,\dots, n+k$. Moreover, the corresponding first $n$ eigenvectors converge to the principal directions of the operator $\tr_\perp\III =\oS_{\bH}-\oR$, and the last $k$ eigenvectors to those of $\bH\otimes\bH + 2\,\tr_\parallel\III$.
\end{theorem}
\begin{proof}
For generic $\VV$ the manifold is again approximated by its tangent space as $\bs{X}=[\bx,\bze]^T$, which produces no contribution to the normal block at leading order $\cO(\vep^{n+2})$. Choosing the tangent orthonormal basis to be aligned with the principal axis of the ellipsoid, and changing variables so that $x^\mu=y^\mu\ell_\mu$, the tangent block becomes an integration over a ball:
\begin{align*}
	[C(\Cyll(\vep,\VV))]^{\mu\nu} & = \int_{\bx^TA\cdot A^T\bx\leq\vep^2} x^\mu x^\nu\, d^n\bx = \int_{\sum_\mu y_\mu^2\leq 1} y^\mu y^\nu \ell_\mu\ell_\nu\,\prod^n_{\al=1}\ell_\al\; d^n\by \\
	& = \del_{\mu\nu}\frac{\vep^{n+2}}{n+2}\ell_\mu\ell_\nu V_n(1)\prod^n_{\al=1}\ell_\al + \cO(\vep^{n+3}).
\end{align*}
Thus, the covariance matrix leading term is proportional to $\diag(\ell^2_1,\dots,\ell^2_n,0,\dots,0)$, which has limit eigenvectors corresponding to the first $n$ eigenvalues spanning $T_p\cM$, and the other $k$ eigenvectors spanning $N_p\cM$, by an straightforward extension to lemma \ref{lem:matrix} at order $\vep^{2}$.

For $\VV=T_p\cM$, we shall compute the integrals of the matrix blocks $[x^\mu x^\nu]_{\mu,\nu=1}^n$, and $[f^if^j]_{i,\, j=1}^k$, so the next-to-leading order elements of those blocks will suffice to obtain the eigenvalues and limit eigenvectors by the results of the previous lemma. The tangent block is:
\begin{align*}
[C(\Cyll(\vep))]^{\mu\nu} & = \int_{B^{(n)}(\vep)}x^\mu x^\nu\sqrt{\det g(\bx)}\, d^n\bx \\ 
& = \int_{\SS^{n-1}}d\,\SS\int^\vep_0\rho^{n+1}\cx^\mu\cx^\nu\left( 1+\frac{1}{2}\sum^k_{i=1}\sum^n_{\al=1}\left[\sum^n_{\bet=1}\ka^j_{\al\bet}\rho\,\cx^\bet \right]^2 +\cO(x^3) \right)d\rho \\
& = \frac{\vep^{n+2}}{n+2}\int_{\SS^{n-1}}\cx^\mu\cx^\nu d\,\SS +\frac{\vep^{n+4}}{2(n+4)}\sum_{i=1}^k\sum_{\al=1}^n\sum_{\bet,\ga}^n\ka^i_{\al\bet}\ka^i_{\al\ga}\int_{\SS^{n-1}}\cx^\mu\cx^\nu\cx^\bet\cx^\ga d\,\SS  +\cO(\vep^{n+6}),
\end{align*}
and the last integral is only nonzero for the following combination of indices using the notation in the appendix
\begin{equation}\label{eq:contr4}
\int_{\SS^{n-1}}\cx^\mu\cx^\nu\cx^\bet\cx^\ga \; d\,\SS = 
C_4\contraction{(}{\mu}{}{\nu}\contraction{(}{\mu}{\nu}{\bet}\contraction{(}{\mu}{\nu\bet}{\ga}(\mu\nu\bet\ga) +
C_{22}\left[
\contraction{(}{\mu}{}{\nu}\bcontraction{(\mu\mu}{\bet}{}{\ga}(\mu\nu\bet\ga) + 
\contraction{(}{\mu}{\nu}{\ga}\bcontraction{(\mu}{\nu}{\bet}{\ga}(\mu\nu\bet\ga) + 
\contraction{(}{\mu}{\nu\bet}{\ga}\bcontraction{(\mu}{\nu}{}{\bet}(\mu\nu\bet\ga)
\right].
\end{equation}
This simplifies the sums using the relationship between $C_4, C_{22}$ and $C_2$, and writing $(1-\del_{\mu\nu})$ to enforce $\mu\neq\nu$ in the last two terms of $C_{22}$:
\begin{align*}
& \frac{\del_{\mu\nu} C_2\vep^{n+2}}{n+2}+ \frac{C_2\vep^{n+4}}{2(n+2)(n+4)}\sum_{i=1}^k\left[ 3\del_{\mu\nu}\sum_{\al=1}^n (\ka^i_{\al\mu})^2 + \del_{\mu\nu}\sum_{\substack{\al,\,\bet \\ \bet\neq \mu}}^n (\ka^i_{\al\bet})^2 + 2(1-\del_{\mu\nu})\sum_{\al=1}^n\ka^i_{\al\mu}\ka^i_{\al\nu} \right]+... \\
& = V_n(\vep)\frac{\vep^2}{n+2}\del_{\mu\nu}+ \frac{V_n(\vep)\vep^4}{2(n+2)(n+4)}\left[ \del_{\mu\nu}\sum_{i=1}^k\sum_{\al,\bet}^n(\ka^i_{\al\bet})^2 +2\sum_{i=1}^k\sum_{\al=1}^n\ka^i_{\al\mu}\ka^i_{\al\nu} \right]+\cO(\vep^{n+6}) \\
& = \frac{V_n(\vep)\vep^2}{n+2}\del_{\mu\nu}+ \frac{V_n(\vep)\vep^4}{2(n+2)(n+4)}\left( \del_{\mu\nu}\sum_{\al,\,\bet}^n\|\II(\bE_\al,\bE_\bet)\|^2 +2\sum_{\al=1}^n\metric{\II(\bE_\al,\bE_\mu)}{\II(\bE_\al,\bE_\nu)} \right)+\cO(\vep^{n+6}).
\end{align*}
The component expression of equations \ref{eq:trNcomp} and \ref{eq:trTotal} identify this block matrix at order $\cO(\vep^{n+4})$ as the matrix elements of the operator $[(\tr_\parallel\tr_\perp\III)\Id_n + 2\tr_\perp\III]$ in our chosen orthonormal basis, whose eigenvalues are then by lemma \ref{lem:matrix} the next-to-leading order contribution to the first $n$ eigenvalues of $C(\Cyll(\vep))$, and whose eigenvectors are the limit eigenvectors of $C(\Cyll(\vep))$.

We perform now the integration of the normal block, which truncated to leading order yields:
\begin{align*}
[C(\Cyll(\vep))]^{ij} =\int_{B^{(n)}(\vep)}f^i(\bx)f^j(\bx) d^n\bx +... =  \int_{\SS^{n-1}}d\,\SS\int^\vep_0 \frac{\rho^{n+3}}{4}d\rho\sum^n_{\al,\bet}\sum^n_{\ga,\,\del}\ka^i_{\al\bet}\ka^j_{\ga\del}\cx^\al\cx^\bet\cx^\ga\cx^\del + \cO(\vep^{n+6})
\end{align*}
where the angular integral is only nonzero in the same cases as in equation \ref{eq:contr4} above, but with the indices relabeled accordingly. This again simplifies every summation by matching the combination of indices and using the relations among the constants:
\begin{align*}
& [C(\Cyll(\vep))]^{ij} = \frac{\vep^{n+4}}{4(n+4)}\left[ C_4\sum_{\al=1}^n\ka^i_{\al\al}\ka^j_{\al\al} +C_{22}\left( \sum^n_{\substack{\al,\,\ga \\ \al\neq\ga}}\ka^i_{\al\al}\ka^j_{\ga\ga} + 2\sum^n_{\substack{\al,\,\bet \\ \al\neq\bet}}\ka^i_{\al\bet}\ka^j_{\al\bet} \right)\right] + \cO(\vep^{n+6}) \\
& = \frac{C_2\,\vep^{n+4}}{4(n+2)(n+4)}\left[ 3\sum_{\al=1}^n \mr{II}^i(\bE_\al,\bE_\al)\mr{II}^j(\bE_\al,\bE_\al) + \sum^n_{\substack{\al,\,\ga \\ \al\neq\ga}}\mr{II}^i(\bE_\al,\bE_\al)\mr{II}^j(\bE_\ga,\bE_\ga) +\right. \\ 
& \qquad\qquad\qquad\qquad\qquad\qquad\qquad\qquad\qquad\qquad\qquad \left. + 2\sum^n_{\substack{\al,\,\bet \\ \al\neq\bet}}\mr{II}^i(\bE_\al,\bE_\bet)\mr{II}^j(\bE_\al,\bE_\bet)  \right] + \cO(\vep^{n+6}),
\end{align*}
in which the first sum precisely completes the elements missing from the other two
$$
= \frac{V_n(\vep)\vep^2}{4(n+2)(n+4)}\left[ \left(\sum_{\al=1}^n\mr{II}^i(\bE_\al,\bE_\al)\right)\left(\sum_{\ga=1}^n\mr{II}^j(\bE_\ga,\bE_\ga)\right) + 2\sum_{\al,\bet}^n\mr{II}^i(\bE_\al,\bE_\bet)\mr{II}^j(\bE_\al,\bE_\bet) \right] + \cO(\vep^{n+6}).
$$
In this last expression we clearly identify the components $[\bH\otimes\bH]^{ij}$, and those of $2\,\tr_\parallel\III$ using the definition of $\bH$ and equation \ref{eq:trTcomp}.
\end{proof}

We shall see below that the spherical covariance matrix has the same normal eigenvalues, to leading order, as the cylindrical case above. In \cite{solis1993,solis2000} these were expressed as an average of the squares of the curvatures of curves inside the manifold $\cM$. Therefore, our previous computation provides an explicit formula for this interpretation of the normal eigenvalues.

\begin{corollary}
	Let $\cM$ be an $n$-dimensional submanifold of Euclidean space $\RR^{n+k}$, then the first generalized curvatures $\ka(\ga,\bx,\bN_j)$ of curves $\ga\subset\cM$ passing through $p$ with tangent vector $\bx$ and principal normal vectors any of the eigenvectors $\bN_j,\; j=1,\dots,k$, of $[\,\bH\otimes\bH + 2\,\tr_\parallel\III\,]$, integrate to:
\begin{equation}
	\frac{1}{V_n(\vep)}\int_{B^{(n)}(\vep)}\ka^2(\ga,\bx,\bN_j)\, d^n\bx = \frac{\vep^4}{(n+2)(n+4)}  \lbd_j[\,\bH\otimes\bH + 2\,\tr_\parallel\III\,].
\end{equation}
In particular:
	\begin{equation}
		\sum_{j=1}^k\frac{1}{V_n(\vep)}\int_{B^{(n)}(\vep)}\ka^2(\ga,\bx,\bN_j)\, d^n\bx = \frac{3\|\bH\|^2-2\cR}{(n+2)(n+4)}\vep^4 .
	\end{equation}
\end{corollary}

%%%%%%%%%%%%%%%%%%%%%%%%%%%%%%%%%%%%%%%%%%%%%%%%%%%%%%%%%%%%%%%%%%%%%%
%%%%%%%%%%%%%%%%%%%%%%%%%%%%%%%%%%%%%%%%%%%%%%%%%%%%%%%%%%%%%%%%%%%%%%
%%%%%%%%%%%%%%%			5. SPHERICAL COVARIANCE
%%%%%%%%%%%%%%%%%%%%%%%%%%%%%%%%%%%%%%%%%%%%%%%%%%%%%%%%%%%%%%%%%%%%%%
%%%%%%%%%%%%%%%%%%%%%%%%%%%%%%%%%%%%%%%%%%%%%%%%%%%%%%%%%%%%%%%%%%%%%%

\section{Spherical Covariance Analysis}\label{sec:SphCov}

The difference between the cylindrical and spherical intersection domains for a graph manifold lies in the irregular projection onto the tangent space: by definition the cylinder is the extension in the normal directions of the ball $B^{(n)}_p(\vep)\subset T_p\cM$, so the points of the graph manifold satisfy $\|\projj{[\bx,\bs{f}(\bx)]^T}{T_p\cM}\|=\|\bx\| \leq\vep$, and thus the integration region is a perfect ball. However, in the spherical case the domain of integration is $\|\bx\|^2+\|\bs{f}(\bx)\|^2 \leq \vep^2$, which is nontrivial and in general cannot be parametrized exactly. One can nevertheless apply the same procedure as done originally in \cite{pottmann2007} and \cite{alvarez2018a} to find the leading order corrections to the ball domain.

\begin{lemma}\label{lem:boundary}
	For $\vep>0$ small enough so that $\cM$ is a graph manifold over $T_p\cM$, using cylindrical coordinates, the radial parametric equation of a point $\bs{X}=[\rho\cx^1,\dots,\rho\cx^n, f^1(\rho\bs{\cx}),\dots,f^k(\rho\bs{\cx})]^T$ in $\dd D_p(\vep)=\cM\cap \SS^{n}_p(\vep)$, is
\begin{equation}
	r(\bs{\cx}):=\rho(\cx_1,\dots,\cx_n) = \vep -\frac{K(\bs{\cx})^2}{8}\vep^3 +\cO(\vep^4) ,
\end{equation}
	where $\bs{\cx}\in\SS^{n-1}\subset T_p\cM$, and
\begin{equation}\label{eq:geoAcc}
	K(\bs{\cx})^2 :=\|\II(\bs{\cx},\bs{\cx}) \|^2=\sum_{i=1}^k\sum_{\al,\bet}^n\sum_{\ga,\del}^n\;\ka^i_{\al\bet}\ka^i_{\ga\del}\,\cx^\al\cx^\bet\cx^\ga\cx^\del
\end{equation}
is the square of the ambient space acceleration of a geodesic curve of $\cM$ with tangent $\bs{\cx}$ at $p$.
\end{lemma}
\begin{proof}
A point of the spherical boundary satisfies $\|\bx\|^2+\sum_{i=1}^k (f^i(\bx))^2=\vep^2$. Since $\|\bx\|^2=\rho^2$, and $f^i(\bx)=\frac{1}{2}\sum^n_{\al,\bet}\ka^i_{\al\bet}x^\al x^\bet+\cO(x^3)$, it is immediate that
 $$
 \rho^2+\frac{1}{4}\rho^4\sum_{i=1}^k\left(\sum^n_{\al,\bet}\ka^i_{\al\bet}\cx^\al\cx^\bet\right)^2 -\vep^2 = \cO(\rho^5).
 $$
Defining $K(\bs{\cx})^2$ as the coefficient of $\frac{\rho^4}{4}$, we can solve the equation to order four to get
	$$
	\rho^2 = \frac{2}{K(\bs{\cx})^2}\left(-1 +\sqrt{1+K(\bs{\cx})^2\vep^2}\right) = \vep^2 -\frac{1}{4}K(\bs{\cx})^2\vep^4+\cO(\vep^6),
	$$
whose square root yields the result. Note that the actual error may be of order four because this could contribute at order fie upon squaring the expression, which is the order neglected in the original equation.
In our chosen orthonormal basis at $p$, we have that $\II(\bs{\cx},\bs{\cx})= \sum_{i=1}^k\left(\sum^n_{\al,\bet}\ka^i_{\al\bet}\cx^\al \cx^\bet\right)\bN_i$, and this is precisely the ambient space acceleration of a geodesic of $\cM$, cf. \cite[ch. 4, Cor. 10]{oneil1983}.
\end{proof}

\begin{proposition}
	The $n$-dimensional volume of the spherical component is
\begin{equation}
	V(D_p(\vep)) = V_n(\vep)\left[ 1+\frac{\vep^2}{8(n+2)}\;(2\,\tr\III-\|\bH\|^2) +\cO(\vep^3) \right]
\end{equation}
where $2\,\tr\III-\|\bH\|^2=\|\bH\|^2-2\cR$.
\end{proposition}
\begin{proof}
In contrast to the proof of the cylindrical domain, the radial integration introduces new angular corrections due to $r(\bs{\cx})$:
\begin{align*}
	V(D_p(\vep)) & = \int_{\SS^{n-1}}d\,\SS\int^{r(\bs{\cx})}_0 \rho^{n-1}\sqrt{\det g(\rho\bs{\cx})}\; d\rho \\
	& = \int_{\SS^{n-1}}\frac{r(\bs{\cx})^n}{n}d\,\SS + \int_{\SS^{n-1}}\frac{r(\bs{\cx})^{n+2}}{2(n+2)}\sum_{i=1}^k\sum_{\al,\bet,\ga}^n\ka^i_{\al\bet}\ka^i_{\al\ga}\cx^\bet\cx^\ga\;d\,\SS + \cO(\vep^{n+3}),
\end{align*}
the second integral is the same to leading order as in the cylindrical case, hence
\begin{align*}
	& = \int_{\SS^{n-1}}d\,\SS\,\frac{\vep^n}{n}\left[ 1-n\frac{K(\bs{\cx})^2}{8}\vep^2 +\cO(\vep^3)\right] + \frac{V_n(\vep)\,\vep^2}{2(n+2)}\tr\III  + \cO(\vep^{n+3}) \\
	& = V_n(\vep) -\frac{\vep^{n+2}}{8}\sum_{i=1}^k\sum^n_{\al,\bet}\sum^n_{\ga,\del}\;\ka^i_{\al\bet}\ka^i_{\ga\del}\int_{\SS^{n-1}}\cx^\al\cx^\bet\cx^\ga\cx^\del \; d\,\SS  + \frac{V_n(\vep)\,\vep^2}{2(n+2)}\tr\III  + \cO(\vep^{n+3}),
\end{align*}
where the integral is only nonzero as in equation \ref{eq:contr4}, so 
\begin{align*}
	& = V_n(\vep) -\frac{C_2\,\vep^{n+2}}{8(n+2)}\sum_{i=1}^k\left[3\sum_{a=1}^n(\ka^i_{\al\al})^2 +\sum^n_{\substack{\al,\ga \\ \al\neq\ga}}\ka^i_{\al\al}\ka^i_{\ga\ga} + 2\sum^n_{\substack{\al,\bet \\ \al\neq\bet}}(\ka^i_{\al\bet})^2 \right] + \frac{V_n(\vep)\,\vep^2}{2(n+2)}\tr\III  + \cO(\vep^{n+3}) \\
	& = V_n(\vep)\left[1 +\frac{\vep^{2}}{8(n+2)}\left(4\,\tr\III -\sum_{i=1}^k\sum^n_{\al,\ga}\ka^i_{\al\al}\ka^i_{\ga\ga} - 2\sum_{i=1}^k\sum^n_{\al,\bet}(\ka^i_{\al\bet})^2 \right) + \cO(\vep^3)\right]
\end{align*}
Now, the first set of sums in the braces is $\metric{\sum_\al\II(\bE_\al,\bE_\al)}{\sum_\ga\II(\bE_\ga,\bE_\ga)} =\|\bH\|^2$, and the second set is $\tr\III$.
\end{proof}
\begin{remark}
	Notice that it is not known the dependence of the error generated by the irregular radius $r(\bs{\cx})$, $\cO(\vep^{n+3})$ in the previous proof, and whether it cancels at that order upon spherical integration, so the spherical component invariants may have error terms at lower order than the cylindrical ones.
\end{remark}

\begin{proposition}
	The barycenter of the spherical component is to leading order the same as for the cylindrical component:
\begin{equation}
	\bb(D_p(\vep)) = [ \;\bze,\; \frac{\vep^2}{2(n+2)}\;\bH\;]^T +\cO(\vep^4).
\end{equation}
\end{proposition}
\begin{proof}
	The new contributions from $r(\bs{\cx})$ to the cylindrical computations are at least of the same order, $\cO(\vep^4)$, as the overall error. 
\end{proof}

The covariance integral invariants for the spherical domain were obtained for hypersurfaces in \cite{alvarez2018a} by performing the computations in the basis of principal and normal directions. In arbitrary codimension, the different osculating paraboloids of $f^i(\bx), i=1,\dots, k$, cannot be diagonalized simultaneously to a common basis in general. The amount of terms and simplifications needed in this general case is of much higher complexity than for hypersurfaces but, nevertheless, an analogous result for the eigenvalue decomposition obtains.

\begin{theorem}\label{MainTh2}
	 Let $\lbd_l[\cdot]$ denote taking the $l$-th eigenvalue of a linear operator at $p$, or of its associated bilinar form with respect to the metric. Then the eigenvalues of the covariance matrix of the spherical component are:
\begin{align}
		&  \lbd_\mu (D_p(\vep)) =V_n(\vep)\left[\frac{\vep^2}{n+2}+\frac{\vep^4}{8(n+2)(n+4)}\lbd_\mu[\, (2\,\tr\III -\|\bH\|^2)\Id_n -4\oS_{\bH} \,] +\;\cO(\vep^5)\right] \\
& \lbd_j (D_p(\vep)) =V_n(\vep)\left[\frac{\vep^4}{2(n+2)(n+4)} \lbd_j[\,\tr_\parallel\III -\frac{1}{n+2}\bH\otimes\bH\,] +\cO(\vep^6)\right]
\end{align}
for all $\mu=1,\dots, n$, and $j=n+1,\dots, n+k$. Moreover, the corresponding first $n$ eigenvectors converge to the principal directions of the Weingarten operator at $\bH$, i.e., $\oS_{\bH}$, and the last $k$ eigenvectors to those of $[\tr_\parallel\III -\frac{1}{n+2}\bH\otimes\bH]$.
\end{theorem}
\begin{proof}
From lemma \ref{lem:matrix} again, only the tangent and normal blocks need to be computed. Now, however, the covariance matrix is taken with respect to the barycenter, so there is an extra matrix contribution from the tensor product,
$$
C(D_p(\vep)) =\int_{D_p(\vep)} \bs{X}\otimes\bs{X}\;\text{dVol} - \int_{D_p(\vep)} \bs{X}\otimes\bb\;\text{dVol},
$$
because the other two products cancel each other upon integration. From the proof of the barycenter formula, this integral is to leading order:
\begin{align*}
	\int_{D_p(\vep)} \bs{X}\otimes\bb\;\text{dVol} = V(D_p(\vep))\bb \otimes\bb = 
\left( \begin{array}{@{}c|c@{}}
	  \cO(\vep^{n+8})_{n\times n}
      & \cO(\vep^{n+6})_{n\times k} \\
   \cmidrule[0.4pt]{1-2}
   \cO(\vep^{n+6})_{k\times n} & \frac{V_n(\vep)\vep^4}{4(n+2)^2}\bH\otimes\bH \\
\end{array} \right)
\end{align*}
There is no difference in the normal block computations of this covariance matrix and the cylindrical case proved before, since the corrections coming from $r(\bs{\cx})$ are $\cO(\vep^{n+6})$. Thus, subtracting the barycenter contribution:
$$
\frac{V_n(\vep)\vep^4}{4(n+2)(n+4)}(\bH\otimes\bH + 2\,\tr_\parallel\III ) - \frac{V_n(\vep)\vep^4}{4(n+2)^2}\bH\otimes\bH = \frac{V_n(\vep)\vep^4}{2(n+2)(n+4)} (\,\tr_\parallel\III -\frac{1}{n+2}\bH\otimes\bH ).
$$

For the tangent block, the number of correction terms due to the spherical domain irregularities with respect to the cylindrical case makes a substantial contribution at $\cO(\vep^{n+4})$:
\begin{align*}
& [C(D_p(\vep))]^{\mu\nu} = \int_{\SS^{n-1}}d\,\SS\int^{r(\bs{\cx})}_0\rho^{n+1}\cx^\mu\cx^\nu\left( 1+\frac{1}{2}\sum^k_{i=1}\sum^n_{\al=1}\left[\sum^n_{\bet=1}\ka^i_{\al\bet}\rho\,\cx^\bet \right]^2 +\cO(x^3) \right)d\rho \\
& = \frac{\vep^{n+2}}{n+2}\left[ \del_{\mu\nu}C_2-(n+2)\int_{\SS^{n-1}}\cx^\mu\cx^\nu\frac{K(\bs{\cx})^2\vep^2}{8}d\,\SS +\cO(\vep^3) \right] \\
& \qquad\qquad\qquad\qquad\qquad\qquad\qquad\qquad +\frac{\vep^{n+4}}{2(n+4)}\sum_{i=1}^k\sum_{\al,\bet,\ga}^n\ka^i_{\al\bet}\ka^i_{\al\ga}\int_{\SS^{n-1}}\cx^\mu\cx^\nu\cx^\bet\cx^\ga d\,\SS  +... \\
& = \del_{\mu\nu}\frac{V_n(\vep)\vep^2}{n+2}+\frac{\vep^{n+4}}{2(n+4)}\sum_{i=1}^k\left[ \sum_{\al=1}^n\sum_{\bet,\ga}^n\ka^i_{\al\bet}\ka^i_{\al\ga}C_{(\mu\nu\bet\ga)} -\frac{n+4}{4}\sum_{\al,\bet}^n\sum_{\ga,\del}^n\ka^i_{\al\bet}\ka^i_{\ga\del}C_{(\mu\nu\al\bet\ga\del)} \right] +\cO(\vep^{n+5}),
\end{align*}
where we have made use of equation \ref{eq:geoAcc}, and written $C_{(\al\bet\dots)}$ for the integral over $\SS^{n-1}$ of the monomial product $\cx^\al\cx^\bet\dots$, (notice here the indices are not exponents but coordinate components). The first summation simplifies again with equation \ref{eq:contr4} to yield the cylindrical tangent block, but the other set of sums comprises the 31 spherical integrals of all possible monomials of degree six:
\begin{align*}
& \qquad\qquad\qquad\qquad C_{(\mu\nu\al\bet\ga\del)} = \int_{\SS^{n-1}}\cx^\mu\cx^\nu\cx^\al\cx^\bet\cx^\ga\cx^\del d\,\SS = C_6\contraction{(}{\mu}{}{\nu}\contraction{(}{\mu}{\nu}{\al}\contraction{(}{\mu}{\nu\al}{\bet}\contraction{(}{\mu}{\nu\al\bet}{\ga}\contraction{(}{\mu}{\nu\al\bet\ga}{\del}(\mu\nu\al\bet\ga\del) + \\ 
& \\
& + C_{24}\left[\;
\bcontraction{(}{\mu}{}{\nu}\contraction{(\mu\nu}{\al}{}{\bet}\contraction{(\mu\nu}{\al}{\bet}{\ga}\contraction{(\mu\nu}{\al}{\bet\ga}{\del}(\mu\nu\al\bet\ga\del) +
\bcontraction{(}{\mu}{\nu}{\al}\contraction{(\mu}{\nu}{\al}{\bet}\contraction{(\mu}{\nu}{\al\bet}{\ga}\contraction{(\mu}{\nu}{\al\bet\ga}{\del}(\mu\nu\al\bet\ga\del) + 
\bcontraction{(}{\mu}{\nu\al}{\bet}\contraction{(\mu}{\nu}{}{\al}\contraction{(\mu}{\nu}{\al\bet}{\ga}\contraction{(\mu}{\nu}{\al\bet\ga}{\del}(\mu\nu\al\bet\ga\del) + 
\bcontraction{(}{\mu}{\nu\al\bet}{\ga}\contraction{(\mu}{\nu}{}{\al}\contraction{(\mu}{\nu}{\al}{\bet}\contraction{(\mu}{\nu}{\al\bet\ga}{\del}(\mu\nu\al\bet\ga\del) + 
\bcontraction{(}{\mu}{\nu\al\bet\ga}{\del}\contraction{(\mu}{\nu}{}{\al}\contraction{(\mu}{\nu}{\al}{\bet}\contraction{(\mu}{\nu}{\al\bet}{\ga}(\mu\nu\al\bet\ga\del) +
\bcontraction{(\mu}{\nu}{}{\al}\contraction{(}{\mu}{\nu\al}{\bet}\contraction{(}{\mu}{\nu\al\bet}{\ga}\contraction{(}{\mu}{\nu\al\bet\ga}{\del}(\mu\nu\al\bet\ga\del)+
\bcontraction{(\mu}{\nu}{\al}{\bet}\contraction{(}{\mu}{\nu}{\al}\contraction{(}{\mu}{\nu\al\bet}{\ga}\contraction{(}{\mu}{\nu\al\bet\ga}{\del}(\mu\nu\al\bet\ga\del) +  \right. \\
& 
\bcontraction{(\mu}{\nu}{\al\bet}{\ga}\contraction{(}{\mu}{\nu}{\al}\contraction{(}{\mu}{\nu\al}{\bet}\contraction{(}{\mu}{\nu\al\bet\ga}{\del}(\mu\nu\al\bet\ga\del) +
\bcontraction{(\mu}{\nu}{\al\bet\ga}{\del}\contraction{(}{\mu}{\nu}{\al}\contraction{(}{\mu}{\nu\al}{\bet}\contraction{(}{\mu}{\nu\al\bet}{\ga}(\mu\nu\al\bet\ga\del) +
\bcontraction{(\mu\nu}{\al}{}{\bet}\contraction{(}{\mu}{}{\nu}\contraction{(}{\mu}{\nu\al\bet}{\ga}\contraction{(}{\mu}{\nu\al\bet\ga}{\del}(\mu\nu\al\bet\ga\del) +
\bcontraction{(\mu\nu}{\al}{\bet}{\ga}\contraction{(}{\mu}{}{\nu}\contraction{(}{\mu}{\nu\al}{\bet}\contraction{(}{\mu}{\nu\al\bet\ga}{\del}(\mu\nu\al\bet\ga\del) +
\bcontraction{(\mu\nu}{\al}{\bet\ga}{\del}\contraction{(}{\mu}{}{\nu}\contraction{(}{\mu}{\nu\al}{\bet}\contraction{(}{\mu}{\nu\al\bet}{\ga}(\mu\nu\al\bet\ga\del) +
\bcontraction{(\mu\nu\al}{\bet}{}{\ga}\contraction{(}{\mu}{}{\nu}\contraction{(}{\mu}{\nu}{\al}\contraction{(}{\mu}{\nu\al\bet\ga}{\del}(\mu\nu\al\bet\ga\del) % \qquad\quad\quad 
+ \left.
\bcontraction{(\mu\nu\al}{\bet}{\ga}{\del}\contraction{(}{\mu}{}{\nu}\contraction{(}{\mu}{\nu}{\al}\contraction{(}{\mu}{\nu\al\bet}{\ga}(\mu\nu\al\bet\ga\del) +
\bcontraction{(\mu\nu\al\bet}{\ga}{}{\del}\contraction{(}{\mu}{}{\nu}\contraction{(}{\mu}{\nu}{\al}\contraction{(}{\mu}{\nu\al}{\bet}(\mu\nu\al\bet\ga\del)\; \right] \\
& \\
& + C_{222}\left[\; 
\bcontraction{(}{\mu}{}{\nu}\contraction{(\mu\nu}{\al}{}{\bet}\contraction{(\mu\nu\al\bet}{\ga}{}{\del}(\mu\nu\al\bet\ga\del) +
\right.
\bcontraction{(}{\mu}{}{\nu}\contraction{(\mu\nu}{\al}{\bet}{\ga}\bcontraction{(\mu\nu\al}{\bet}{\ga}{\del}(\mu\nu\al\bet\ga\del) +
\bcontraction{(}{\mu}{}{\nu}\contraction{(\mu\nu}{\al}{\bet\ga}{\del}\bcontraction{(\mu\nu\al}{\bet}{}{\ga}(\mu\nu\al\bet\ga\del) +
\bcontraction{(}{\mu}{\nu}{\al}\contraction{(\mu}{\nu}{\al}{\bet}\contraction{(\mu\nu\al\bet}{\ga}{}{\del}(\mu\nu\al\bet\ga\del) +
\bcontraction{(}{\mu}{\nu}{\al}\contraction{(\mu}{\nu}{\al\bet}{\ga}\bcontraction{(\mu\nu\al}{\bet}{\ga}{\del}(\mu\nu\al\bet\ga\del) +
\bcontraction{(}{\mu}{\nu}{\al}\contraction{(\mu}{\nu}{\al\bet\ga}{\del}\bcontraction{(\mu\nu\al}{\bet}{}{\ga}(\mu\nu\al\bet\ga\del) +
\bcontraction{(}{\mu}{\nu\al}{\bet}\contraction{(\mu}{\nu}{}{\al}\contraction{(\mu\nu\al\bet}{\ga}{}{\del}(\mu\nu\al\bet\ga\del) + \\
& 
\bcontraction{(}{\mu}{\nu\al}{\bet}\contraction{(\mu}{\nu}{\al\bet}{\ga}\contraction[2ex]{(\mu\nu}{\al}{\bet\ga}{\del}(\mu\nu\al\bet\ga\del) +
\bcontraction{(}{\mu}{\nu\al}{\bet}\contraction{(\mu}{\nu}{\al\bet\ga}{\del}\contraction[2ex]{(\mu\nu}{\al}{\bet}{\ga}(\mu\nu\al\bet\ga\del) +
\bcontraction{(}{\mu}{\nu\al\bet}{\ga}\contraction{(\mu}{\nu}{}{\al}\contraction{(\mu\nu\al}{\bet}{\ga}{\del}(\mu\nu\al\bet\ga\del) +
\bcontraction{(}{\mu}{\nu\al\bet}{\ga}\contraction{(\mu}{\nu}{\al}{\bet}\contraction[2ex]{(\mu\nu}{\al}{\bet\ga}{\del}(\mu\nu\al\bet\ga\del) +
\bcontraction{(}{\mu}{\nu\al\bet}{\ga}\contraction{(\mu}{\nu}{\al\bet\ga}{\del}\contraction[2ex]{(\mu\nu}{\al}{}{\bet}(\mu\nu\al\bet\ga\del) +
\bcontraction{(}{\mu}{\nu\al\bet\ga}{\del}\contraction{(\mu}{\nu}{}{\al}\contraction{(\mu\nu\al}{\bet}{}{\ga}(\mu\nu\al\bet\ga\del) +
\bcontraction{(}{\mu}{\nu\al\bet\ga}{\del}\contraction{(\mu}{\nu}{\al}{\bet}\contraction[2ex]{(\mu\nu}{\al}{\bet}{\ga}(\mu\nu\al\bet\ga\del) +
\left.
\bcontraction{(}{\mu}{\nu\al\bet\ga}{\del}\contraction{(\mu}{\nu}{\al\bet}{\ga}\contraction[2ex]{(\mu\nu}{\al}{}{\bet}(\mu\nu\al\bet\ga\del) \right]
%\bcontraction{}{}{}{}\contraction{}{}{}{}\contraction{}{}{}{}(\mu\nu\al\bet\ga\del) +
\end{align*}
Each of these contractions are only nonzero when the connected indices are equal, and at the same time different from the indices of the other connected groups, for instance:
$$
\sum_{\al,\bet}^n\sum_{\ga,\del}^n\ka^i_{\al\bet}\ka^i_{\ga\bet}\bcontraction{(}{\mu}{}{\nu}\contraction{(\mu\nu}{\al}{}{\bet}\contraction{(\mu\nu\al\bet}{\ga}{}{\del}(\mu\nu\al\bet\ga\del) = 
\del_{\mu\nu}\sum^n_{\al\neq\mu}\sum^n_{\substack{\ga\neq\mu \\ \ga\neq\al}} \ka^i_{\al\al}\ka^i_{\ga\ga}.
$$
Matching all the indices in this way for each of the terms just found, and taking into account the relation of $C_6,C_{24}$ and $C_{222}$ to $C_2$ in the appendix, we take out a common factor $\frac{C_2}{4(n+2)}$, and abbreviate the sum notation to produce all the terms of order $\cO(\vep^{n+4})$:
\begin{align*}
& [C(D_p(\vep))]^{\mu\nu} =\frac{ \del_{\mu\nu}V_n(\vep)\vep^2}{n+2} + \frac{C_2\,\vep^{n+4}}{8(n+2)(n+4)} \sum_{i}\left[4\del_{\mu\nu}\sum_{\substack{\al,\,\bet \\ \bet\neq\mu}}(\ka^i_{\al\bet})^2 + 8\cancel{\del}_{\mu\nu}\sum_{\al}\ka^i_{\al\mu}\ka^i_{\al\nu} + 12\del_{\mu\nu}\sum_{\al}(\ka^i_{\al\nu})^2  \right. \\
& -15\del_{\mu\nu}(\ka^i_{\nu\nu})^2 -3\left\{\,\del_{\mu\nu}\sum_{\al\neq\mu}(\ka^i_{\al\al})^2 +\cancel{\del}_{\mu\nu}( \ka^i_{\mu\nu}\ka^i_{\nu\nu} + \ka^i_{\nu\mu}\ka_{\nu\nu}+\ka^i_{\nu\nu}\ka^i_{\mu\nu} + \ka^i_{\nu\nu}\ka^i_{\nu\mu} + \ka^i_{\nu\mu}\ka^i_{\mu\mu} + \ka^i_{\mu\nu}\ka^i_{\mu\mu} \right. \\
& + \ka^i_{\mu\mu}\ka^i_{\nu\mu} + \ka^i_{\mu\mu}\ka^i_{\mu\nu} ) +\del_{\mu\nu}\left( \sum_{\al\neq\mu}\ka^i_{\al\al}\ka^i_{\nu\nu} + \sum_{\al\neq\mu}(\ka^i_{\al\nu})^2 +\sum_{\al\neq\mu}(\ka^i_{\al\nu})^2 + \sum_{\bet\neq\mu}(\ka^i_{\nu\bet})^2 + \sum_{\bet\neq\mu}(\ka^i_{\nu\bet})^2+ \right. \\
& \left.\left. \sum_{\ga\neq\mu}\ka^i_{\ga\ga}\ka^i_{\nu\nu}\right)\right\} - \del_{\mu\nu}\left( \sum_{\al\neq\mu}\sum_{\ga\neq\mu,\al}\ka^i_{\al\al}\ka^i_{\ga\ga} + \sum_{\al\neq\mu}\sum_{\bet\neq\mu,\al}(\ka^i_{\al\bet})^2 + \sum_{\al\neq\mu}\sum_{\bet\neq\mu,\al}(\ka^i_{\al\bet})^2  \right) - \cancel{\del}_{\mu\nu}\left\{ \sum_{\ga\neq\mu,\nu}\ka^i_{\mu\nu}\ka^i_{\ga\ga} \right. \\
& +\sum_{\bet\neq\mu,\nu}\ka^i_{\mu\bet}\ka^i_{\nu\bet} + \sum_{\bet\neq\mu,\nu}\ka^i_{\mu\bet}\ka^i_{\bet\nu} + \sum_{\ga\neq\mu,\nu}\ka^i_{\nu\mu}\ka^i_{\ga\ga} + \sum_{\al\neq\mu,\nu}\ka^i_{\al\mu}\ka^i_{\nu\al}+ \sum_{\al\neq\mu,\nu}\ka^i_{\al\mu}\ka^i_{\al\nu} +\sum_{\bet\neq\mu,\nu}\ka^i_{\nu\bet}\ka^i_{\mu\bet} \\
& +\left.\left. \sum_{\al\neq\mu,\nu}\ka^i_{\al\nu}\ka^i_{\mu\al} + \sum_{\al\neq\mu,\nu}\ka^i_{\al\al}\ka^i_{\mu\nu} + \sum_{\bet\neq\mu,\nu}\ka^i_{\nu\bet}\ka^i_{\bet\mu} + \sum_{\al\neq\mu,\nu}\ka^i_{\al\nu}\ka^i_{\al\mu} + \sum_{\al\neq\mu,\nu}\ka^i_{\al\al}\ka^i_{\nu\mu}\right\}\right] + \cO(\vep^{n+5})
\end{align*}
Many of the resulting summations are the same after relabeling and using $\ka^i_{\al\bet}=\ka^i_{\bet\al}$, so they can be gathered into common factors:
\begin{align*}
& [C(D_p(\vep))]^{\mu\nu} = \del_{\mu\nu}\frac{V_n(\vep)\vep^2}{n+2} + \frac{V_n(\vep)\vep^{4}}{8(n+2)(n+4)} \sum_{i}\left[ 4\del_{\mu\nu}\sum_{\al,\bet}(\ka^i_{\al\bet})^2 + 8\sum_{\al}\ka^i_{\al\mu}\ka^i_{\al\nu} - 15\del_{\mu\nu}(\ka^i_{\nu\nu})^2  \right. \\
& -3\del_{\mu\nu}\sum_{\al\neq\mu}(\ka^i_{\al\al})^2 -12(1-\del_{\mu\nu})\ka^i_{\mu\nu}(\ka^i_{\mu\mu}+\ka^i_{\nu\nu}) -6\del_{\mu\nu}\sum_{\al\neq\mu}\ka^i_{\al\al}\ka^i_{\nu\nu} -12\del_{\mu\nu}\sum_{\al\neq\mu}(\ka^i_{\al\nu})^2  \\
& \left.-\del_{\mu\nu}\sum_{\al\neq\mu}\sum_{\ga\neq\al,\mu}\ka^i_{\al\al}\ka^i_{\ga\ga} -2\del_{\mu\nu}\sum_{\al\neq\mu}\sum_{\bet\neq\al,\mu}(\ka^i_{\al\bet})^2 - (1-\del_{\mu\nu})\left( 4\ka^i_{\mu\nu}\sum_{\al\neq\mu,\nu}\ka^i_{\al\al} + 8\sum_{\al\neq\mu,\nu}\ka^i_{\al\mu}\ka^i_{\nu\al} \right)\right]
\end{align*}
for which regrouping terms and completing some sums will clarify the simplifications below,
\begin{align*}
& = \del_{\mu\nu}\frac{V_n(\vep)\vep^2}{n+2} + \frac{V_n(\vep)\vep^{4}}{8(n+2)(n+4)} \sum_{i}\left[ 8\sum_{\al}\ka^i_{\al\mu}\ka^i_{\al\nu} - 12\ka^i_{\mu\nu}(\ka^i_{\mu\mu}+\ka^i_{\nu\nu}) - 4\ka^i_{\mu\nu}\sum_{\al\neq\mu,\nu}\ka^i_{\al\al}  \right. \\
& - 8\sum_{\al\neq\mu,\nu}\ka^i_{\al\mu}\ka^i_{\nu\al} + \del_{\mu\nu}\left\{ 4\sum_{\al,\,\bet}(\ka^i_{\al\bet})^2 -3\sum_{\al\neq\mu}(\ka^i_{\al\al})^2 + 21(\ka^i_{\mu\mu})^2 -2\ka^i_{\mu\mu}\sum_{\al\neq\mu}\ka^i_{\al\al} -12\sum_{\al}(\ka^i_{\al\mu})^2 \right. \\
& \left.\left. -\sum_{\al\neq\mu}\sum_{\ga\neq\al,\mu}\ka^i_{\al\al}\ka^i_{\ga\ga} -2\sum_{\al\neq\mu}\sum_{\bet\neq\al,\mu}(\ka^i_{\al\bet})^2 + 8\sum_{\al\neq\mu}(\ka^i_{\al\mu})^2 \right\} \right] +\cO(\vep^{n+5}).
\end{align*}
Some terms inside the curly braces complement the missing elements of other summations:
$$
21(\ka^i_{\mu\mu})^2 -2\ka^i_{\mu\mu}\sum_{\al\neq\mu}\ka^i_{\al\al} -12\sum_{\al}(\ka^i_{\al\mu})^2 +8\sum_{\al\neq\mu}(\ka^i_{\al\mu})^2 = 15(\ka^i_{\mu\mu})^2 -2\ka^i_{\mu\mu}\sum_{\al}\ka^i_{\al\al}-4\sum_{\al}(\ka^i_{\al\mu})^2,
$$
and
$$
-3\sum_{\al\neq\mu}(\ka^i_{\al\al})^2 -\sum_{\al\neq\mu}\sum_{\ga\neq\al,\mu}\ka^i_{\al\al}\ka^i_{\ga\ga} -2\sum_{\al\neq\mu}\sum_{\bet\neq\al,\mu}(\ka^i_{\al\bet})^2 = -\sum_{\al,\ga\neq\mu}\ka^i_{\al\al}\ka^i_{\ga\ga} -2\sum_{\al,\bet\neq\mu}(\ka^i_{\al\bet})^2.
$$
Now, notice that this last type of double sum decomposes as follows
$$
-\sum_{\al,\ga\neq\mu}[\;\cdot\;]_{\al\ga} = -\sum_{\al,\,\ga}[\;\cdot\;]_{\al\ga} +\sum_{\substack{\ga \\ \al=\mu}}[\;\cdot\;]_{\al\ga}+\sum_{\substack{\al \\ \ga=\mu}}[\;\cdot\;]_{\al\ga} -[\;\cdot\;]_{\mu\mu},
$$
therefore, the right hand side of the previous two equations complement each other:
\begin{align*}
& [C(D_p(\vep))]^{\mu\nu} = \frac{\del_{\mu\nu}V_n(\vep)\vep^2}{n+2} + \frac{V_n(\vep)\vep^{4}}{8(n+2)(n+4)} \sum_{i}\left[ 8\sum_{\al}\ka^i_{\al\mu}\ka^i_{\al\nu} - 12\ka^i_{\mu\nu}(\ka^i_{\mu\mu}+\ka^i_{\nu\nu}) - 4\ka^i_{\mu\nu}\sum_{\al\neq\mu,\nu}\ka^i_{\al\al}  \right. \\
&\qquad \left. - 8\sum_{\al\neq\mu,\nu}\ka^i_{\al\mu}\ka^i_{\nu\al} + \del_{\mu\nu}\left\{ 4\sum_{\al,\,\bet}(\ka^i_{\al\bet})^2 + 12(\ka^i_{\mu\mu})^2 -\sum_{\al,\,\ga}\ka^i_{\al\al}\ka^i_{\ga\ga} -2\sum_{\al,\,\bet}(\ka^i_{\al\bet})^2 \right\}\right] +\cO(\vep^{n+5}).
\end{align*}
To simplify further, use $12(\ka^i_{\mu\mu})^2$ to complete the remaining sums and cancel terms:
$$
8\sum_{\al}\ka^i_{\al\mu}\ka^i_{\nu\al} - 8\ka^i_{\mu\nu}(\ka^i_{\mu\mu}+\ka^i_{\nu\nu}) -8\sum_{\al\neq\mu,\nu}\ka^i_{\al\mu}\ka^i_{\nu\al} + 8(\ka^i_{\mu\mu})^2\del_{\mu\nu} = 0,
$$
and
$$
- 4\ka^i_{\mu\nu}(\ka^i_{\mu\mu}+\ka^i_{\nu\nu}) -4\ka^i_{\mu\nu}\sum_{\al\neq\mu,\nu}\ka^i_{\al\al} + 4(\ka^i_{\mu\mu})^2\del_{\mu\nu} = -4\ka^i_{\mu\nu}\sum_{\al}\ka^i_{\al\al}. 
$$
Finally, all these computations lead us to the simple expression:
\begin{align*}
[C(D_p(\vep))]^{\mu\nu} & = \frac{\del_{\mu\nu}V_n(\vep)\vep^2}{n+2} + \frac{V_n(\vep)\vep^{4}}{8(n+2)(n+4)} \sum_{i}\left[\del_{\mu\nu}\left\{ 2\sum_{\al,\,\bet}(\ka^i_{\al\bet})^2 - (H^i)^2  \right\}-4\ka^i_{\mu\nu}H^i \right]+\cO(\vep^{n+5})
\end{align*}
where 
$$\sum_i\ka^i_{\mu\nu}H^i = \metric{\II(\bE_\mu,\bE_\nu)}{\bH}=\metric{\oS_{\bH}\,\bE_\mu}{\bE_\nu},$$
and 
$$\sum_i(2\sum_{\al,\,\bet}(\ka^i_{\al\bet})^2 - (H^i)^2)=2\,\tr\III -\|\bH\|^2,$$ 
identify the covariance tangent block to be the matrix of the Weingarten operator at the mean curvature, plus a constant, in the orthonormal basis chosen.
\end{proof}

%%%%%%%%%%%%%%%%%%%%%%%%%%%%%%%%%%%%%%%%%%%%%%%%%%%%%%%%%%%%%%%%%%%%%%
%%%%%%%%%%%%%%%%%%%%%%%%%%%%%%%%%%%%%%%%%%%%%%%%%%%%%%%%%%%%%%%%%%%%%%
%%%%%%%%%%%%%%%			6. DESCRIPTORS
%%%%%%%%%%%%%%%%%%%%%%%%%%%%%%%%%%%%%%%%%%%%%%%%%%%%%%%%%%%%%%%%%%%%%%
%%%%%%%%%%%%%%%%%%%%%%%%%%%%%%%%%%%%%%%%%%%%%%%%%%%%%%%%%%%%%%%%%%%%%%

\section{Curvature Descriptors}\label{sec:descrip}

Curvature descriptors in terms of the covariance eigenvalues were introduced in \cite{pottmann2007} for surfaces and in \cite{alvarez2018a} for hypersurfaces. A limit formula for the ratio of the eigenvalues was found for curves \cite{alvarez2017} to establish a direct relationship between the local covariance analysis of a domain containing the point $p$ and the Frenet-Serret curvature information at $p$, which in the case of curves completely determines the curve locally up to rigid motion \cite[Th. 2.13]{kuhnel2006differential}. The two main theorems of the present work generalize this type of result to general submanifolds by directly taking the limits of the covariance matrix eigenvalues.

\begin{corollary}
 Writing $\lbd_\mu(p,\vep)$ for the tangent eigenvalues of the cylindrical covariance matrix $C(\Cyll(\vep))$, they satisfy the asymptotic ratio
\begin{equation}
		\lim_{\vep\rightarrow 0}V_n(\vep)\frac{\lbd_\mu(p,\vep)-\lbd_\nu(p,\vep)}{\lbd_\mu(p,\vep)\lbd_\nu(p,\vep)} = \frac{n+2}{n+4}\left(\, \lbd_\mu[\tr_\perp\III] - \lbd_\nu[\tr_\perp\III] \,\right),
\end{equation}
	and the normal eigenvalues satisfy
\begin{equation}
		 \lim_{\vep\rightarrow 0}\frac{V_n(\vep)}{\lbd_\mu(p,\vep)\lbd_\nu(p,\vep)}\sum_{j=n+1}^{n+k}\lbd_j(p,\vep) = \frac{n+2}{4(n+4)}\left(\, \|\bH\|^2 +2\,\tr\III \, \right),
\end{equation}
	for any $\mu,\nu=1,\dots, n$. Let $\tilde{\lbd}_\mu(p,\vep)$ denote the eigenvalues in the case of the spherical domain covariance matrix, $C_p(D_p(\vep))$, then the corresponding limits are
\begin{equation}
		\lim_{\vep\rightarrow 0}V_n(\vep)\frac{\tilde{\lbd}_\mu(p,\vep)-\tilde{\lbd}_\nu(p,\vep)}{\tilde{\lbd}_\mu(p,\vep)\tilde{\lbd}_\nu(p,\vep)} = \frac{n+2}{2(n+4)}\left(\, \tilde{\lbd}_\nu[\oS_{\bH}] - \tilde{\lbd}_\mu[\oS_{\bH}] \,\right),
\end{equation}
	and
\begin{equation}
		 \lim_{\vep\rightarrow 0}\frac{V_n(\vep)}{\tilde{\lbd}_\mu(p,\vep)\tilde{\lbd}_\nu(p,\vep)}\sum_{j=n+1}^{n+k}\tilde{\lbd}_j(p,\vep) = \frac{n+2}{2(n+4)}\left(\, \tr\III -\frac{1}{n+2}\|\bH\|^2 \, \right).
\end{equation}
\end{corollary}

Now we focus on smooth hypersurfaces in $\RR^{n+1}$. Theorems \ref{MainTh} and \ref{MainTh2} provide formulas to extract curvature estimators at scale from the eigenvalues of the covariance matrices. Doing this analysis on a hypersurface furnishes descriptors at scale of the principal curvatures, and the principal and normal directions. As explained in \cite{alvarez2018a}, for an embedded Riemannian manifold $\cM\subset\RR^{n+k}$, of general codimension $k$, it can always be projected down locally to $k$ hypersurfaces by choosing $k$ linearly independent orthogonal directions $\bN_j$ of its normal space, and project the points to the linear subspace $T_p\cM\oplus\langle \bN_j\rangle$. Approximations of the principal curvatures and directions of these hypersurfaces are sufficient to build an estimator of the second fundamental form of the original manifold and, by Gau{\ss} equation \ref{th:Gauss}, get in turn a descriptor of its Riemann curvature tensor.

\begin{example}
For a smooth hypersurface $\cS$, there is only one unit normal vector $\bs{n}$ at every point $p\in\cS$, up to orientation. Choosing $\{\bE_\mu\}_{\mu=1}^n$ as the orthonormal basis of the tangent space given by the principal directions at $p$, the components of the third fundamental form are:
\begin{equation}
	\metric{\III(\bE_\mu,\bE_\nu)\bN}{\bN}=\metric{\oS\,\bE_\mu}{\oS\,\bE_\nu}=\metric{\oS^2\,\bE_\mu}{\bE_\nu} = \ka^2_\mu\del_{\mu\nu} = \tr_\perp\III(\bE_\mu,\bE_\nu).
\end{equation}
The tangent trace components are
\begin{equation}
	\metric{\tr_\parallel\III\,\bN}{\bN} = \tr(\oS^2) = \sum_{\mu=1}^n \ka^2_\mu =  H^2 -\cR,
\end{equation}
that coincides with the total trace, $\tr\III = \sum_{\mu=1}^n \ka^2_\mu= H^2-\cR$.

The limit eigenvectors of either $C(\Cyll(\vep))$ or $C(D_p(\vep))$ yield a local adapted orthonormal frame $\langle\bE_1,\dots,\bE_n\rangle\oplus\langle\bN\rangle$ of $\RR^{n+1}$ that precisely singles out the tangent and normal spaces at every generic point. If the principal curvatures at $p$ are of different absolute value, this basis exactly points in the principal and normal directions. The tangent eigenvalues of the cylindrical covariance matrix $C(\Cyll(\vep))$ satisfy
\begin{equation}
	\lim_{\vep\rightarrow 0}V_n(\vep)\frac{\lbd_\mu(p,\vep)-\lbd_\nu(p,\vep)}{\lbd_\mu(p,\vep)\lbd_\nu(p,\vep)} = \frac{n+2}{n+4}(\;\ka^2_\mu(p) -\ka^2_\nu(p)\; ),
\end{equation}
	and the normal eigenvalue has
\begin{equation}
		 \lim_{\vep\rightarrow 0}V_n(\vep)\frac{\lbd_{n+1}(p,\vep)}{\lbd_\mu(p,\vep)\lbd_\nu(p,\vep)} = \frac{n+2}{4(n+4)}\left(\; 3H^2(p)-2\cR(p) \; \right),
\end{equation}
	for any $\mu,\nu=1,\dots, n$. 
For the spherical covariance matrix $C(D_p(\vep))$, the limits are
\begin{equation}
	\lim_{\vep\rightarrow 0}V_n(\vep)\frac{\tilde{\lbd}_\mu(p,\vep) -\tilde{\lbd}_\nu(p,\vep)}{\tilde{\lbd}_\mu(p,\vep)\tilde{\lbd}_\nu(p,\vep)} = \frac{n+2}{2(n+4)}[\ka_\nu(p) -\ka_\mu(p)]H(p),
\end{equation}
	and
\begin{equation}
	\lim_{\vep\rightarrow 0}V_n(\vep)\frac{\tilde{\lbd}_{n+1}(p,\vep)}{\tilde{\lbd}_\mu(p,\vep)\tilde{\lbd}_\nu(p,\vep)} = \frac{n+2}{2(n+4)}\left[\frac{n+1}{n+2}H^2(p) -\cR(p)\right].
\end{equation}
\end{example}

The known terms of the series expansion of the eigenvalue decomposition of the covariance matrices can be inverted to extract the curvature descriptors upon truncations of the series. In the spherical case, one recovers the results and descriptors already obtained in \cite{alvarez2018a}.

\begin{corollary}
	Let us write $\lbd(p,\vep)\equiv\lbd(D_p(\vep)), V_p(\vep)\equiv V(D_p(\vep))$ for the integral invariants of a spherical domain on a hypersurface $\cS$, then the corresponding curvature descriptors at scale $\vep>0$ and point $p\in\cS$, for any $\mu=1,\dots,n$, are:
\begin{align}
	\cR(D^+_p(\vep)) & = 2(n+2)^2(n+4)\frac{\lbd_{n+1}(p,\vep)}{n\,\vep^4\, V_n(\vep)} - \frac{8(n+1)(n+2)}{n\,\vep^2}\left(\frac{V_p(\vep)}{V_n(\vep)} - 1 \right) \\[3mm]
	H(D^+_p(\vep)) & = (\pm)\sqrt{ 4(n+2)^2(n+4)\frac{\lbd_{n+1}(p,\vep)}{n\,\vep^4 V_n(\vep)} +\frac{8(n+2)^2}{n\,\vep^2}\left(1-\frac{V_p(\vep)}{V_n(\vep)} \right) }, \\[3mm]
	\ka_\mu(D^+_p(\vep)) & = \frac{2(n+2)}{\vep^2 H(D^+_p(\vep))}\left[ \frac{V_p(\vep)}{V_n(\vep)}+\frac{n+4}{\vep^2}\left( \frac{\vep^2}{n+2}-\frac{\lbd_\mu(p,\vep)}{V_n(\vep)} \right) -1 \right],
\end{align}
where the overall sign can be chosen by fixing a normal orientation from $$(\pm)=\text{\emph{sgn}}\langle\,\bE_{n+1}(D_p(\vep)),\,\bb(D_p(\vep))\,\rangle .$$ The eigenvectors $\bE_\mu(D_p(\vep))$ and $\bE_{n+1}(D_p(\vep))$ are descriptors of the principal and normal directions respectively. The errors are:
\begin{align*}
	& |H^2(p)-H^2(D_p(\vep))|\leq\cO(\vep), \qquad |\cR(p)-\cR(D_p(\vep)) |\leq\cO(\vep), \qquad |\ka^2_\mu(p)-\ka^2_\mu(D_p(\vep))|\leq\cO(\vep).
\end{align*}
\end{corollary}

The cylindrical domain descriptors may determine in general the squares of the principal curvatures with better truncation error than their spherical domain counterparts.

\begin{corollary}
	Denote $\lbd(p,\vep)\equiv\lbd(\Cyll(\vep)), V_p(\vep)\equiv V(\Cyll(\vep))$ the integral invariants of a cylindrical domain on a hypersurface $\cS$, then the corresponding curvature descriptors at scale $\vep>0$ and point $p\in\cS$, for any $\mu=1,\dots,n$, are:
\begin{align}
	\cR(\Cyll(\vep)) & = \frac{2(n+2)}{\vep^2}\left[ \frac{2(n+4)}{\vep^2}\frac{\lbd_{n+1}(p,\vep)}{V_n(\vep)} + 3\left( 1-\frac{V_p(\vep)}{V_n(\vep)} \right) \right] \\[3mm]
	H(\Cyll(\vep)) & = (\pm)\sqrt{ \frac{2(n+2)}{\vep^2}\left[ \frac{2(n+4)}{\vep^2}\frac{\lbd_{n+1}(p,\vep)}{V_n(\vep)} + 2\left( 1-\frac{V_p(\vep)}{V_n(\vep)} \right) \right] }, \\[3mm]
	\ka^2_\mu(\Cyll(\vep)) & = \frac{n+2}{\vep^2}\left[ \frac{n+4}{\vep^2}\left( \frac{\lbd_\mu(p,\vep)}{V_n(\vep)} -\frac{\vep^2}{n+2} \right) -\frac{V_p(\vep)}{V_n(\vep)} +1 \right],
\end{align}
where the overall sign can be chosen by fixing a normal orientation from $$(\pm)=\text{\emph{sgn}}\langle\,\bE_{n+1}(\Cyll(\vep)),\,\bb(\Cyll(\vep))\,\rangle .$$ The eigenvectors $\bE_\mu(\Cyll(\vep))$ and $\bE_{n+1}(\Cyll(\vep))$ are descriptors of the principal and normal directions respectively. The truncation errors are: 
\begin{align*}
	& |H^2(p)-H^2(\Cyll(\vep))|\leq\cO(\vep^2), \; |\cR(p)-\cR(\Cyll(\vep)) |\leq\cO(\vep^2), \; |\ka^2_\mu(p)-\ka^2_\mu(\Cyll(\vep))|\leq\cO(\vep^2).
\end{align*}
\end{corollary}
\begin{proof}
	Solving for the next-to-leading order term in the volume formula \ref{eq:volCyl}, and for the normal eigenvalue in equation \ref{eq:norEVDCyl}, we get a system of two equations $H^2-\cR = A(\vep),\; 3H^2 - 2\cR = B(\vep)$, whose solution is $H^2 = B-2A$ and $\cR=B-3A$, where
\begin{align*}
	& A(\vep) = \frac{2(n+2)}{\vep^2}\left( \frac{V_p(\vep)}{V_n(\vep)} -1 \right) +\cO(\vep^2), \qquad
	 B(\vep) = \frac{4(n+2)(n+4)}{\vep^4}\frac{\lbd_{n+1}(p,\vep)}{V_n(\vep)} + \cO(\vep^2).
\end{align*}
Finally, solving for $\ka^2_\mu$ from the tangent eigenvalue equation \ref{eq:tanEVDCyl}, and using $A(\vep)=\sum_\al \ka^2_\al$, the last formula obtains.
\end{proof}

The spherical descriptors can be used to determine the relative signs of the principal curvatures, and the cylindrical descriptors can be used to estimate with higher precision the absolute value of the principal curvatures.

%%%%%%%%%%%%%%%%%%%%%%%%%%%%%%%%%%%%%%%%%%%%%%%%%%%%%%%%%%%%%%%%%%%%%%
%%%%%%%%%%%%%%%%%%%%%%%%%%%%%%%%%%%%%%%%%%%%%%%%%%%%%%%%%%%%%%%%%%%%%%

%%%%%%%%%%%%%%%%%%%%%%%%%%%%%%%%%%%%%%%%%%%%%%%%%%%%%%%%%%%%%%%%%%%%%%
%%%%%%%%%%%%%%%%%%%%%%%%%%%%%%%%%%%%%%%%%%%%%%%%%%%%%%%%%%%%%%%%%%%%%%
%%%%%%%%%%%%%%%        6. CONCLUSIONS
%%%%%%%%%%%%%%%%%%%%%%%%%%%%%%%%%%%%%%%%%%%%%%%%%%%%%%%%%%%%%%%%%%%%%%
%%%%%%%%%%%%%%%%%%%%%%%%%%%%%%%%%%%%%%%%%%%%%%%%%%%%%%%%%%%%%%%%%%%%%%

\section{Conclusions}
\label{sec:conclusions}

We have used the exponential map to propose a generalization of the multi-scale integral invariants determined by performing Principal Component Analysis in small regions of $n$-dimensional submanifolds inside a general $(n+k)$-dimensional Riemannian manifold. The kernel domains studied for Riemannian manifolds embedded in Euclidean space were determined by the manifold intersection with higher-dimensional cylinders and balls in the ambient space. The volume of these regions expands with scale as the volume of the $n$-dimensional ball plus second order corrections proportional to the mean curvature and scalar curvature of the submanifold at the center point. We have also introduced a generalization of the classical third fundamental form to any codimension and showed how it relates to the Weingarten and Ricci operators and the Ricci equation. Then, the covariance analysis of the region point-set was found to have eigenvalues encoding curvature in terms of the third fundamental form; in particular, the first $n$ eigenvalues are related to those of the normal trace of the third fundamental form operator and the corresponding eigenvectors converge to its principal directions, whereas the last $k$ eigenvalues and eigenvectors are related to the tangent trace of this tensor. In the case of the spherical domain the tangent eigenvalues and eigenvectors of the covariance matrix are related to the Weingarten operator at the mean curvature vector. For hypersurfaces, these eigenvalues provide a method to estimate the principal curvatures and principal directions, furnishing descriptors for general submanifolds via the analysis of their independent hypersurface projections. These results show how local integral invariants relate to the same geometric information traditionally characterized by differential-geometric invariants.

%
%%%%%%%%%%%%%%%%%%%%%%%%%%%%%%%%%%%%%%%%%%%%%%%%%%%%%%%%%%%%%%%%%%%%%%%
%%%%%%%%%%%%%%%%%%%%%%%%%%%%%%%%%%%%%%%%%%%%%%%%%%%%%%%%%%%%%%%%%%%%%%%
%%%%%%%%%%%%%%%%        A. APPENDIX: INTEGRALS
%%%%%%%%%%%%%%%%%%%%%%%%%%%%%%%%%%%%%%%%%%%%%%%%%%%%%%%%%%%%%%%%%%%%%%%
%%%%%%%%%%%%%%%%%%%%%%%%%%%%%%%%%%%%%%%%%%%%%%%%%%%%%%%%%%%%%%%%%%%%%%%
%
%
\appendix
\section{Integration of Monomials over Spheres}\label{sec:appendix}
Let $\bx =[x^1,\dots,x^n ]^T\in\RR^n$, and denote unit the sphere and ball of radius $\vep$ in $\RR^n$ by:
$$
\SS^{n-1}=\{\bx\in\RR^n : \|\bx\|=1 \},\quad B^{n}(\vep)=\{\bx\in\RR^n : \|\bx\|\leq\vep \}.
$$
General spherical coordinates $(r,\phi_1,\dots,\phi_{n-1})$ are given by $r=\|\bx\|$, where $\cx^\mu :=x^\mu / r\in\SS^{n-1}$.
\begin{definition}
 For any integers $p_1,\dots,p_n\in\{0,1,2,\dots\}$, the integrals of the monomials $(x^1)^{p_1}\cdots (x^n)^{p_n}$ over the unit sphere and the ball of radius $\vep$ are denoted by:
\begin{equation}\label{eq:constants}
	C^{(n)}_{p_1\dots p_n}=\int_{\SS^{n-1}} (x^1)^{p_1}\cdots (x^n)^{p_n}\; d\,\SS^{n-1},\quad\quad
	D^{(n)}_{p_1\dots p_n}=\int_{B^n(\vep)} (x^1)^{p_1}\cdots (x^n)^{p_n}\; d^nB.
\end{equation}
where $d\,\SS^{n-1}$ is the Euclidean measure on the sphere and $d^nB=dx^1\cdots dx^n=r^{n-1}dr\;d\,\SS^{n-1}$.
\end{definition}

The following formula is crucial to the computations of the present paper, cf. \cite{folland2001}.

\begin{theorem}\label{th:folland}
	Let $b_i=\frac{1}{2}(p_i+1)$, then the values of the integrals \ref{eq:constants} over spheres are
\begin{equation}
	C^{(n)}_{p_1\dots p_n}=\begin{cases}
		0, & \text{ if some $p_i$ is odd,} \\ \displaystyle
		2\frac{\Gamma(b_1)\Gamma(b_2)\cdots\Gamma(b_n)}{\Gamma(b_1+b_2+\cdots +b_n)}, & \text{ if all $p_i$ are even},
	\end{cases}
\end{equation}
and the integrals over balls become
\begin{equation}
	D^{(n)}_{p_1\dots p_n}=\frac{\vep^{n+p_1+\cdots+p_n}}{n+p_1+\cdots+p_n}\; C^{(n)}_{p_1\dots p_n}.
\end{equation}
\end{theorem}

\begin{example}\label{ex:constants}
	We shall need the relations among integrals of monomials of even powers:
\begin{align*}
	& C_{2} = \int_{\SS^{n-1}} (x^1)^{2}\; d\,\SS= 2\frac{\Gamma(\frac{3}{2})\Gamma(\frac{1}{2})^{n-1}}{\Gamma(\frac{3}{2}+\frac{n-1}{2})}=\frac{\pi^{n/2}}{\Gamma(\frac{n}{2}+1)}, \\[1mm]
	& C_{22}= \int_{\SS^{n-1}} (x^1)^{2}(x^2)^2\; d\,\SS = \frac{1}{n+2}\, C_2, \\[1mm]
	& C_{4}= \int_{\SS^{n-1}} (x^1)^{4}\; d\,\SS = \frac{3}{n+2}\, C_2 = 3\, C_{22}, \\[1mm]
	& C_{222}= \int_{\SS^{n-1}} (x^1)^{2}(x^2)^2(x^3)^2\; d\,\SS = \frac{1}{(n+2)(n+4)}\, C_2, \\[1mm]
	& C_{24}= \int_{\SS^{n-1}} (x^1)^{2}(x^2)^4\; d\,\SS = \frac{3}{(n+2)(n+4)}\, C_2 = 3\, C_{222}, \\[1mm]
	& C_{6}= \int_{\SS^{n-1}} (x^1)^{6}\; d\,\SS = \frac{15}{(n+2)(n+4)}\, C_2 = 15\, C_{222}.
\end{align*}
The volume of a ball of radius $\vep$, and the area of the unit sphere satisfy:
\begin{equation*}
	V_n(\vep)=\text{Vol}(B^n(\vep)) = \vep^n\, C_2,\quad\quad S_{n-1}=\text{Area}(\SS^{n-1})= n\, C_2.
\end{equation*}
\end{example}

The integral of a general combination of coordinates depends on the superindices involved, which must not be confused with exponents. For instance
$$
\int_{\SS^{n-1}}\cx^\mu\cx^\nu\cx^\bet\cx^\ga \; d\,\SS = 
C_4\contraction{(}{\mu}{}{\nu}\contraction{(}{\mu}{\nu}{\bet}\contraction{(}{\mu}{\nu\bet}{\ga}(\mu\nu\bet\ga) +
C_{22}\left[
\contraction{(}{\mu}{}{\nu}\bcontraction{(\mu\mu}{\bet}{}{\ga}(\mu\nu\bet\ga) + 
\contraction{(}{\mu}{\nu}{\ga}\bcontraction{(\mu}{\nu}{\bet}{\ga}(\mu\nu\bet\ga) + 
\contraction{(}{\mu}{\nu\bet}{\ga}\bcontraction{(\mu}{\nu}{}{\bet}(\mu\nu\bet\ga)
\right]
$$
is the general value of the integral of any product of 4 coordinates, that can be all equal to produce $C_4$, or be a couple of different pairs to result in $C_{22}$. We introduce the following notation:
$$
\contraction{(}{\mu}{}{\nu}\bcontraction{(\mu\mu}{\bet}{}{\ga}(\mu\nu\bet\ga) = \del_{\mu\nu}\,\del_{\bet\ga}\cancel{\del}_{\mu\bet},
$$
so that the symbol is $1$ only when the connected superindices are equal and the nonconnected superindices are different, and $0$ otherwise, and where $\cancel{\del}_{\mu\bet}:=(1-\del_{\mu\bet})$ is the negation of the Kronecker delta, i.e., nonzero only if $\mu\neq\bet$. An example of order 6 is 
$$
\bcontraction{(}{\mu}{\nu\al\bet}{\ga}\contraction{(\mu}{\nu}{\al\bet\ga}{\del}\contraction[2ex]{(\mu\nu}{\al}{}{\bet}(\mu\nu\al\bet\ga\del) = \del_{\mu\ga}\,\del_{\nu\del}\,\del_{\al\bet}\,\cancel{\del}_{\mu\nu}\,\cancel{\del}_{\mu\al}\,\cancel{\del}_{\nu\al}.
$$

\addtocontents{toc}{\SkipTocEntry}
\section*{Acknowledgments}
We would like to thank Louis Scharf for very helpful discussions. J.\'A.V. would like to thank Miguel Dovale \'Alvarez for many useful discussions during the writing of this paper.

\bibliographystyle{amsplain}
\bibliography{AKP-Integral-Invariants}

\end{document}